\documentclass[12pt]{amsart}
\usepackage{amsmath,amssymb,bm}
\usepackage{graphicx,color}
\usepackage{diagbox}
\usepackage{subfigure}
\newtheorem{theorem}{Theorem}[section]
\newtheorem{lemma}[theorem]{Lemma}

\theoremstyle{definition}

\newtheorem{example}[theorem]{Example}
\theoremstyle{remark}
\newtheorem{remark}[theorem]{Remark}
\numberwithin{equation}{section}
\numberwithin{table}{section}
\numberwithin{figure}{section}
\newcommand{\trinorm}[1]{%
  \left|\mkern-1.5mu\left|\mkern-1.5mu\left|
   #1
  \right|\mkern-1.5mu\right|\mkern-1.5mu\right|
}
\DeclareMathOperator*{\argmin}{argmin}
\begin{document}
\def\d{\displaystyle}
\def\O{\Omega}
\def\p{\partial}
\def\LT{{L_2(\O)}}
\def\cB{\mathcal{B}}
\def\cT{\mathcal{T}}
\def\HB{{H^1_\beta(\O)}}
\def\HO{{H^1(\O)}}
\def\HOz{{H^1_0(\O)}}
\def\ES{\HOz\times\HOz}
\def\FS{V_h\times V_h}
\def\FE#1{V_{#1}\times V_{#1}}
\def\cV{\mathcal{V}}
\def\fC{\mathfrak{C}}
\def\fB{\mathfrak{B}}
\def\FO{\fB_k\fC_k^{-1}\fB_k}
\def\AP{Id_k-I_{k-1}^kP_k^{k-1}}
\def\red{\color{red}}
\def\splus{{\hspace{-1pt}\tiny{+}\!}}
\title[Multigrid methods for saddle point problems: KKT Systems]
{Multigrid methods for saddle point problems:  Karush-Kuhn-Tucker systems}
\author{Susanne C. Brenner}
\address{Susanne C. Brenner, Department of Mathematics and Center for
Computation and Technology, Louisiana State University, Baton Rouge}
\email{brenner@math.lsu.edu}
\author{Sijing Liu}
\address{Sijing Liu, Department of Mathematics and Center for Computation and Technology,
 Louisiana State University, Baton Rouge, LA 70803, USA}
\email{sliu42@lsu.edu}
\author{Li-yeng Sung}
 \address{Li-yeng Sung,
 Department of Mathematics and Center for Computation and Technology,
 Louisiana State University, Baton Rouge, LA 70803, USA}
\email{sung@math.lsu.edu}
\thanks{This work was supported in part
 by the National Science Foundation under Grant No.
 DMS-16-20273.}
\date{September 22, 2018}
\begin{abstract}
 We construct multigrid methods for an elliptic distributed optimal control problem
  that are robust with respect to a regularization parameter.  We prove the uniform
  convergence of the $W$-cycle  algorithm and demonstrate the performance of
  $V$-cycle and $W$-cycle algorithms in two and three  dimensions through
 numerical experiments.
\end{abstract}
\maketitle
\section{Introduction}\label{sec:Introduction}
 Let $\O$ be a bounded convex polygonal domain in $\mathbb{R}^d$ ($d=2, 3$),
 $y_d\in \LT$, $\beta$ be a positive constant and
 $(\cdot,\cdot)_\LT$ be the inner product of $\LT$ (or $[\LT]^d$).
  The optimal control problem is to find
\begin{equation}\label{eq:OCP}
 (\bar{y},\bar{u})=\argmin_{(y,u)\in K}\left [
 \frac{1}{2}\|y-y_d\|^2_{\LT}+\frac{\beta}{2}\|u\|^2_\LT\right],
\end{equation}
 where $(y,u)$ belongs to $K\subset H^1_0(\O)\times \LT$ if and only if
\begin{equation}\label{eq:PDEConstraint}
 (\nabla y, \nabla v)_\LT=(u,v)_\LT \quad \forall\, v\in H^1_0(\Omega).
\end{equation}
 Here and throughout the paper we will follow the standard notation for differential operators,
 function spaces and norms that can be found for example in \cite{Ciarlet:1978:FEM,BScott:2008:FEM}.
\par
 The optimal control problem \eqref{eq:OCP}--\eqref{eq:PDEConstraint} has a unique solution
 characterized by the following system of equations (cf. \cite{Lions:1971:OC,Troltzsch:2010:OC}):
\begin{subequations}\label{subeq:Lions}
\begin{alignat}{3}
(\nabla\bar{p},\nabla q)_\LT&=(\bar{y}-y_d,q)_\LT &\qquad& \forall\, q\in\HOz,\label{eq:Lions1}\\
\bar{p}+\beta\bar{u}&=0,\label{eq:Lions2}\\
(\nabla\bar{y},\nabla z)_\LT&=(\bar{u},z)_\LT &\qquad&\forall\,  z\in \HOz,\label{eq:Lions3}
\end{alignat}
\end{subequations}
 where $\bar{p}$ is the (optimal) adjoint state.
 After eliminating $\bar{u}$, we have a symmetric saddle point problem
\begin{subequations}\label{subeq:SPP}
\begin{alignat}{3}
(\nabla\bar{p},\nabla q)_\LT-(\bar{y},q)_\LT&=-(y_d,q)_\LT &\qquad& \forall\, q\in \HOz,
\label{eq:SPP1}\\
-(\bar{p},z)_\LT-\beta(\nabla \bar{y},\nabla z)_\LT&=0  &\qquad&\forall\, z \in \HOz.
\label{eq:SPP2}
\end{alignat}
\end{subequations}
\par
 Note that the system \eqref{subeq:SPP} is unbalanced with respect to $\beta$
 since it only appears  in \eqref{eq:SPP2}.  This can be remedied by the following
 change of variables:
\begin{equation}\label{eq:CV}
\bar{p}=\beta^{\frac{1}{4}}\tilde{p}\quad\text{and}\quad
\bar{y}=\beta^{-\frac{1}{4}}\tilde{y}.
\end{equation}
 The resulting saddle point problem is
\begin{subequations}\label{subeq:BSPP}
\begin{alignat}{3}
 \beta^{\frac{1}{2}}(\nabla \tilde{p},\nabla q)_\LT-(\tilde{y},q)_\LT
 &=-\beta^{\frac{1}{4}}(y_d,q)_\LT
  &\qquad&\forall \,{q}\in \HOz,\label{eq:BSPP1}\\
 -(\tilde{p},{z})_\LT-\beta^{\frac{1}{2}}(\nabla\tilde{y},\nabla{z})_\LT&=0
 &\qquad& \forall\,{z}\in \HOz.\label{eq:BSPP2}
\end{alignat}
\end{subequations}
\par
 The saddle point problem \eqref{subeq:BSPP} can be discretized by a $P_1$ finite element method.
 Our goal is to design multigrid methods for the resulting discrete saddle point problem
 whose performance is independent of the regularization parameter $\beta$.
 The key idea is to use a post-smoother that can be interpreted as a Richardson
 iteration for a symmetric positive definite (SPD) problem that has the same solution
 as the saddle point problem.  Consequently we can exploit the well-known multigrid
 theory for SPD problems \cite{Hackbusch:1985:MMA,MMB:1987:VMT,BZ:2000:AMG}
 in our convergence analysis.  This idea has previously been applied to other
 saddle point problems in
 \cite{BLS:2014:StokesLame,BLS:2017:Oseen,BOS:2018:Darcy}.
\par
 Our multigrid methods belong to the class of all-at-once methods where all the unknowns
 in \eqref{subeq:SPP} are solved simultaneously (cf.
 \cite{BS:2009:MGReview,EG:2011:MG,SSZ:2011:MG,BS:2012:OC,TZ:2013:MGOCP} and the references
 therein).  Multigrid methods that are robust with respect to $\beta$ can also be found in
  the papers \cite{SSZ:2011:MG,TZ:2013:MGOCP}.
 The differences are in the construction of the smoothers and in the norms that measure
 the convergence of the multigrid algorithms.  The smoothing steps in \cite{SSZ:2011:MG,TZ:2013:MGOCP} are
 computationally
 less expensive than the one in the current paper, which requires solving (approximately)
 a diffusion-reaction problem (which however does not affect the $O(n)$ complexity).
   The trade-off is that the convergence of the multigrid algorithm in this paper
  is expressed in terms of the natural energy norm for the continuous problem, while the norms in
  \cite{SSZ:2011:MG,TZ:2013:MGOCP} are different from the energy norm.  A related consequence is that
  the $W$-cycle multigrid algorithms in \cite{SSZ:2011:MG,TZ:2013:MGOCP} cannot take advantage of
  post-smoothing
  and hence their contraction numbers decay at the rate of $O(m^{-1/2})$,
   where $m$ is the number of pre-smoothing
  steps, while the contraction number for our symmetric $W$-cycle
  multigrid algorithm decays at the rate of $O(m^{-1})$, where
  $m$ is the number of pre-smoothing and post-smoothing steps.
\par
 The rest of the paper is organized as follows.  We analyze the saddle point problem
 \eqref{subeq:BSPP}  and the $P_1$ finite element method
 in Section~\ref{sec:P1} and introduce the multigrid algorithms in Section~\ref{sec:MG}.
  We derive smoothing and approximation properties in Section~\ref{sec:SA} that are
  the key ingredients for the convergence analysis of the $W$-cycle algorithm in
  Section~\ref{sec:Convergence}.  Numerical results are
  presented in Section~\ref{sec:Numerics} and we end with some concluding remarks
  in Section~\ref{sec:Conclusions}.
\par
 Throughout this paper, we use $C$
 (with or without subscripts) to denote a generic positive
 constant that is independent of $\beta$ and any mesh
 parameter.
  Also to avoid the proliferation of constants, we use the
   notation $A\lesssim B$ (or $A\gtrsim B$) to
  represent $A\leq \text{(constant)}B$, where the (hidden) positive constant is independent of
   $\beta$ and any mesh parameter. The notation $A\approx B$ is equivalent to
  $A\lesssim B$ and $B\lesssim A$.
%
\section{A $P_1$ Finite Element Method}\label{sec:P1}
 We can express \eqref{subeq:BSPP} concisely as
\begin{equation}\label{eq:ConciseBSPP}
 \cB((\tilde{p},\tilde{y}),({q},{z}))
 =-\beta^{\frac{1}{4}}(y_d,{q})_\LT \qquad \forall\, ({q},{z})\in \ES,
\end{equation}
 where
\begin{equation}\label{eq:BDef}
 \cB((p,y),(q,z))=\beta^{\frac{1}{2}}(\nabla p,\nabla q)_\LT-(y,q)_\LT-(p,z)_\LT
 -\beta^{\frac{1}{2}}(\nabla y,\nabla z)_\LT.
\end{equation}
%
\subsection{Properties of $\cB$}\label{subsec:B}
 We will analyze the bilinear form $\cB(\cdot,\cdot)$ in terms of
 the weighted $H^1$ norm $\|\cdot\|_\HB$ defined by
\begin{equation}\label{eq:HB}
 \|v\|^2_\HB=\|v\|^2_\LT+\beta^{\frac{1}{2}}|v|^2_\HO \qquad\forall\,v\in\HO.
\end{equation}
\par
 Let $(p,y)\in\HO\times\HO$ be arbitrary.
 It follows immediately from \eqref{eq:BDef}, \eqref{eq:HB} and the Cauchy-Schwarz inequality that
\begin{equation}\label{eq:BUpper}
\cB((p,y),(q,z))\le (\|p\|^2_\HB+\|y\|^2_\HB)^{\frac{1}{2}}(\|q\|^2_\HB+\|z\|^2_\HB)^{\frac{1}{2}}.
\end{equation}
 Moreover, a direct calculation shows that
\begin{align}
  \cB((p,y),(p-y,-y-p))&=\beta^\frac12(\nabla p,\nabla(p-y))_\LT-(y,p-y)_\LT\notag\\
      &\hspace{30pt}+(p,y+p)_\LT+\beta^\frac12(\nabla y,\nabla(y+p))_\LT  \label{eq:BId1}  \\
      &=\|p\|_\HB^2+\|y\|_\HB^2,\notag
      \intertext{and we also have, by the parallelogram law,}
     \|p-y\|_\HB^2+\|-y-p\|_\HB^2&=2(\|p\|_\HB^2+\|y\|_\HB^2). \label{eq:BId2}
\end{align}
\par
 It follows from \eqref{eq:BUpper}--\eqref{eq:BId2} that
\begin{align}\label{eq:InfSup}
   &\,(\|p\|^2_\HB+\|y\|^2_\HB)^\frac12\notag\\
  \geq&\, \sup_{(q,z)\in \ES}\frac{\cB((p,y),(q,z))}{(\|q\|^2_\HB
  +\|z\|^2_\HB)^\frac12}\\
  \geq&\, 2^{-\frac12}(\|p\|^2_\HB+\|y\|^2_\HB)^\frac{1}{2} &\quad \forall\,(p,y)\in\ES.
  \notag
\end{align}
%
\subsection{The Discrete Problem}\label{subsec:Discrete}
 Let $\cT_h$ be a simplicial triangulation of $\O$ and $V_h\subset \HOz$ be the $P_1$ finite element space
  associated with $\cT_h$.
 The $P_1$ finite element method for \eqref{eq:ConciseBSPP}
  is to find $(\tilde{p}_h,\tilde{y}_h)\in \FS$ such that
\begin{equation}\label{eq:Discrete}
 \cB((\tilde{p}_h,\tilde{y}_h),({q}_h,{z}_h))=-\beta^\frac14(y_d,{q}_h)_\LT
  \qquad\forall\,({q}_h,{z}_h)\in \FS.
\end{equation}
\par
 For the convergence analysis of the multigrid algorithms, it is necessary to
 consider a more general problem: Find $(p,y)\in \ES$ such that
\begin{equation}\label{eq:GP}
\cB((p,y),(q,z))=(f,q)_\LT+(g,z)_\LT \qquad\forall\,(q,z)\in\ES,
\end{equation}
 where $f,g\in\LT$.  The unique solvability of \eqref{eq:GP} follows immediately from
 \eqref{eq:InfSup}.
\par
 The $P_1$ finite element method for \eqref{eq:GP} is to find
 $(p_h,y_h)\in V_h\times V_h$ such that
\begin{equation}\label{eq:DGP}
\cB((p_h,y_h),(q_h,z_h))=(f,q_h)_\LT+(g,z_h)_\LT \qquad\forall\,(q_h,z_h)\in \FS.
\end{equation}
 Note that \eqref{eq:BUpper}--\eqref{eq:BId2} also yield the following analog of
 \eqref{eq:InfSup}:
\begin{align}\label{eq:DiscreteInfSup}
 &\,(\|p_h\|^2_\HB+\|y_h\|^2_\HB)^\frac12\notag\\
 \geq&\,
 \sup_{(q_h,z_h)\in \FS}\frac{\cB((p_h,y_h),(q_h,z_h))}{(\|q_h\|^2_\HB
   +\|z_h\|^2_\HB)^\frac12}\\
  \geq&\, 2^{-\frac12}(\|p_h\|^2_\HB+\|y_h\|^2_\HB)^\frac12
  &\quad\forall\,(p_h,y_h)\in\FS.\notag
\end{align}
  Therefore the discrete problem \eqref{eq:DGP} is
 uniquely solvable.
\subsection{Error Estimates for \eqref{eq:GP}}\label{subsec:Error}
 From \eqref{eq:BUpper}, \eqref{eq:DiscreteInfSup} and the saddle point theory
 \cite{Babushka:1973:LM,Brezzi:1974:SPP,XZ:2003:BB}, we have the following quasi-optimal
 error estimate.
\begin{lemma}\label{lem:QuasiOptimal}
  Let $(p,y)$ $($resp., $(p_h,y_h))$ be the solution of \eqref{eq:GP} $(resp., \eqref{eq:DGP})$.
  We have
\begin{equation}\label{eq:QuasiOptimal}
  \|p-p_h\|_\HB^2+\|y-y_h\|_\HB^2
  \leq 2\inf_{(q_h,z_h)\in V_h\times V_h}
     \big(\|p-q_h\|_\HB^2+\|y-z_h\|_\HB^2\big).
\end{equation}
\end{lemma}
\par
 In order to convert \eqref{eq:QuasiOptimal} into a concrete error estimate, we need
 the regularity of the solution of \eqref{eq:GP}.
\begin{lemma}\label{lem:Regularity}
 The solution $(p,y)$ of \eqref{eq:GP} belongs to $H^2(\O)\times H^2(\O)$ and
 we have
\begin{equation}\label{eq:Regularity}
 \|\beta^{\frac{1}{2}}p\|_{H^2(\O)}+\|\beta^{\frac{1}{2}}y\|_{H^2(\O)}\leq
  C_{\O}(\|f\|_\LT+\|g\|_\LT).
\end{equation}
\end{lemma}
\begin{proof} We can write \eqref{eq:GP} as
\begin{alignat*}{3}
 (\nabla(\beta^{\frac{1}{2}}p),\nabla q)_\LT&=(y+f,q)_\LT &\qquad&\forall\,q\in \HOz,\\
(\nabla(\beta^{\frac{1}{2}}y),\nabla z)_\LT&=(-p-g,z)_\LT.&\qquad&\forall\,z\in\HOz,
\end{alignat*}
 and hence, by the elliptic regularity theory for convex domains
 \cite{Grisvard:1985:EPN,Dauge:1988:EBV},
\begin{subequations}\label{subeq:Regularity}
\begin{align}
\|\beta^{\frac{1}{2}}p\|_{H^2(\O)}&\leq C_{\Omega}(\|y\|_\LT+\|f\|_\LT),\label{eq:Rp}\\
\|\beta^{\frac{1}{2}}y\|_{H^2(\O)}&\leq C_{\Omega}(\|p\|_\LT+\|g\|_\LT).\label{eq:Ry}
\end{align}
\end{subequations}
\par
 From \eqref{eq:HB}, \eqref{eq:InfSup} and \eqref{eq:GP} we also have
\begin{equation}\label{eq:LTBdd}
  \|p\|_\LT^2+\|y\|_\LT^2 \leq 2 (\|f\|_\LT^2+\|g\|_\LT^2).
\end{equation}
\par
 The estimate \eqref{eq:Regularity} follows from \eqref{subeq:Regularity} and
 \eqref{eq:LTBdd}.
\end{proof}
\par
 We can now derive concrete error estimates for the $P_1$ finite element method for
 \eqref{eq:GP}.
\begin{lemma}\label{lem:ConcreteErrors}
  Let $(p,y)$ $($resp., $(p_h,y_h))$ be the solution of \eqref{eq:GP} $(resp., \eqref{eq:DGP})$.
  We have
\begin{align}
  \|p-p_h\|_\HB+\|y-y_h\|_\HB&\leq C(1+{\beta}^{\frac{1}{2}}h^{-2})^{\frac{1}{2}}
  \beta^{-\frac{1}{2}}  h^2(\|f\|_\LT+\|g\|_\LT), \label{eq:EnergyError}\\
 \|p-p_h\|_\LT+\|y-y_h\|_\LT&\leq C(1+{\beta}^{\frac{1}{2}}h^{-2})\beta^{-1}h^4
 (\|f\|_\LT+\|g\|_\LT),  \label{eq:LTError}
\end{align}
 where the positive constant $C$ is independent of $\beta$ and $h$.
\end{lemma}
\begin{proof} Let $\Pi_h:H^2(\O)\cap \HOz\longrightarrow V_h$ be the
 nodal interpolation operator.  We have the following standard interpolation
 error estimate \cite{Ciarlet:1978:FEM,BScott:2008:FEM}:
\begin{equation}\label{eq:PihEst}
 \|\zeta-\Pi_h\zeta\|_\LT+h|\zeta-\Pi_h\zeta|_\HO\leq Ch^2|\zeta|_{H^2(\O)}
 \qquad\forall\,\zeta\in H^2(\O)\cap\HOz,
\end{equation}
 where the positive constant $C$ only depends on the shape regularity of $\cT_h$.
\par
 The estimate \eqref{eq:EnergyError} follows from \eqref{eq:HB},
 \eqref{eq:QuasiOptimal}, \eqref{eq:Regularity}  and \eqref{eq:PihEst}:
\begin{align*}
 &\|p-p_h\|_\HB^2+\|y-y_h\|_\HB^2
   \leq 2\big(\|p-\Pi_hp\|_\HB^2+\|y-\Pi_hy\|_\HB^2\big)\\
   &\hspace{40pt}\leq 2\big(\|p-\Pi_hp\|_\LT^2+\beta^\frac12|p-\Pi_hp|_\HO^2
      +\|y-\Pi_hy\|_\LT^2+\beta^\frac12|y-\Pi_hy|_\HO^2\big)
   \\
   &\hspace{40pt}\leq C(\beta^{-1}h^4+\beta^{-\frac12}h^2)(\|f\|_\LT^2+\|g\|_\LT^2)\\
   &\hspace{40pt}= C(1+\beta^\frac12 h^{-2})\beta^{-1}h^4(\|f\|_\LT^2+\|g\|_\LT^2).
\end{align*}
\par
 The estimate \eqref{eq:LTError} is established by a duality argument.
 Let $(\xi,\theta)\in \ES$ be defined by
\begin{equation}\label{eq:Duality1}
 \cB((\xi,\theta),(q,z))=(p-p_h,q)_\LT+(y-y_h,z)_\LT \qquad \forall\, (q,z)\in \ES.
\end{equation}
 We have, by \eqref{eq:BUpper}, Lemma~\ref{lem:Regularity} (applied to \eqref{eq:Duality1}),
  \eqref{eq:PihEst}, \eqref{eq:Duality1} and Galerkin orthogonality,
\begin{align*}
 &\|p-p_h\|^2_\LT+\|y-y_h\|^2_\LT=(p-p_h,p-p_h)_\LT+(y-y_h,y-y_h)_\LT\\
  &\hspace{40pt}=\cB((\xi,\theta),(p-p_h,y-y_h))\\
 &\hspace{40pt}=\cB((\xi-\Pi_h\xi,\theta-\Pi_h\theta),(p-p_h,y-y_h))\\
 &\hspace{40pt}\leq(\|\xi-\Pi_h\xi\|^2_\HB+\|\theta-\Pi_h\theta\|^2_\HB)^{\frac{1}{2}}
  (\|p-p_h\|^2_\HB+\|y-y_h\|^2_\HB)^{\frac{1}{2}}\\
  &\hspace{40pt}\leq
  C(1+{\beta}^{\frac{1}{2}}h^{-2})^\frac12\beta^{-\frac12}h^2
  (\|p-p_h\|_\LT^2+\|y-y_h\|_\LT^2)^\frac12\\
      &\hspace{70pt}\times  (\|p-p_h\|^2_\HB+\|y-y_h\|^2_\HB)^{\frac{1}{2}},
\end{align*}
 which together with \eqref{eq:EnergyError} implies \eqref{eq:LTError}.
\end{proof}
\subsection{A $\bm{P_1}$ Finite Element Method  for \eqref{subeq:SPP}}\label{eq:OriginalError}
 The $P_1$ finite element method for \eqref{subeq:SPP} is to find
 $(\bar{p}_h,\bar{y}_h)\in\FS$ such that
\begin{subequations}\label{subeq:OriginalP1}
\begin{alignat}{3}
 (\nabla\bar{p}_h,\nabla q_h)_\LT-(\bar{y}_h,q_h)_\LT&=-(y_d,q_h)_\LT
 &\qquad& \forall \, q_h\in V_h,\\
-(\bar{p}_h,z_h)_\LT-\beta(\nabla \bar{y}_h,\nabla z_h)_\LT&=0
&\qquad& \forall \, z_h \in V_h,
\end{alignat}
\end{subequations}
 which is equivalent to \eqref{eq:Discrete}
 under the change of variables
\begin{equation}\label{eq:CV2}
\bar{p}_h=\beta^{\frac{1}{4}}\tilde{p}_h\quad\text{and}\quad
\bar{y}_h=\beta^{-\frac{1}{4}}\tilde{y}_h.
\end{equation}
\par
 Applying the results in Section~\ref{subsec:Error} to
 \eqref{eq:ConciseBSPP} and \eqref{eq:Discrete},
 we arrive at the following error estimates through the change of variables
 \eqref{eq:CV} and \eqref{eq:CV2}.
\begin{lemma}\label{lem:OriginalErrors}
 Let $(\bar p,\bar y)$ $($resp., $(\bar{p}_h,\bar{y}_h))$ be the solution of \eqref{subeq:SPP}
 $($resp., \eqref{subeq:OriginalP1}$)$.  We have
\begin{align*}
 \|\bar{p}-\bar{p}_h\|_\HB+\beta^{\frac{1}{2}}\|\bar{y}-\bar{y}_h\|_\HB
  &\leq C(1+{\beta}^{\frac{1}{2}}h^{-2})^{\frac{1}{2}}h^2\|y_d\|_\LT,\\
 \|\bar{p}-\bar{p}_h\|_\LT+\beta^{\frac{1}{2}}\|\bar{y}-\bar{y}_h\|_\LT
 &\leq C(1+{\beta}^{\frac{1}{2}}h^{-2})\beta^{-\frac{1}{2}}h^4\|y_d\|_\LT,
\end{align*}
 where the positive constant $C$ is independent of $\beta$ and $h$.
\end{lemma}
\par
 According to Lemma~\ref{lem:OriginalErrors}, the performance of the $P_1$ finite element method
 defined by \eqref{subeq:OriginalP1} will deteriorate as $\beta\downarrow0$.    Indeed it can be
 shown that
\begin{equation*}
 \frac{|\bar p-\bar p_h|_\HO}{|\bar p|_\HO},\;
  \frac{|\bar y-\bar y_h|_\HO}{|\bar y|_\HO}\leq C
     \beta^{-\frac14}h \quad\text{and}\quad
    \frac{\|\bar p-\bar p_h\|_\LT}{\|\bar p\|_\LT},\;
     \frac{\|\bar y-\bar y_h\|_\LT}{\|\bar y\|_\LT}\leq C
    \beta^{-\frac12}h^2
\end{equation*}
 as $h\downarrow0$, where the positive constant $C$ is independent of $\beta$ and $h$.
 This phenomenon is due to
  the mismatch between the homogeneous Dirichlet boundary condition for $y$  and
 the fact that $y_d$ only belongs to $\LT$.
  In the case where $y_d\in\HOz\cap H^2(\O)$, the estimates for the
  asymptotic  relative  errors can be  improved to
\begin{equation*}
  \frac{|\bar p-\bar p_h|_\HO}{|\bar p|_\HO},\; \frac{|\bar y-\bar y_h|_\HO}{|\bar y|_\HO}\leq C
     h \quad\text{and}\quad
    \frac{\|\bar p-\bar p_h\|_\LT}{\|\bar p\|_\LT},\;
    \frac{\|\bar y-\bar y_h\|_\LT}{\|\bar y\|_\LT}\leq C h^2.
\end{equation*}
\par
 The performance of the $P_1$ finite element method is illustrated in the following
 example.
\begin{example}\label{example:BoundaryLayer}
  We solve \eqref{subeq:SPP} by the $P_1$ finite
  element method defined by \eqref{subeq:OriginalP1} on the unit square $\O=(0,1)\times (0,1)$
  for $y_d=1$ and $y_d=x_1(1-x_1)x_2(1-x_2)$.   In both cases the exact solution can be found
  in the form of a double Fourier sine series.
   The relative errors for $h=2^{-6}$ and
  various $\beta$ together with the solution times (in seconds) are displayed in
  Table~\ref{table:P1Errors}.  The numerical solutions are obtained by a full multigrid method
  (cf. \cite[Section~6.7]{BScott:2008:FEM})
  using the symmetric $W$-cycle algorithm from Section~\ref{sec:MG}
  with 2 pre-smoothing and 2 post-smoothing steps,
  where the preconditioner
  $\fC_k^{-1}$ in the  smoothing steps is based on
  a $V(4,4)$ multigrid solve for the boundary value problem \eqref{eq:BVP}.
  The full multigrid iteration at each level is terminated when the relative residual error is
  $\leq 10^{-8}$.
 \end{example}
\begin{table}[h]
\begin{tabular}{|c||c|c||c|c|c|}\hline&&&&&\\[-11pt]
$\beta$& $\d\frac{|\bar{p}-\bar p_h|_\HO}{|\bar{p}|_\HO}$&
$\d\frac{\|\bar{p}-\bar p_h\|_\LT}{\|\bar{p}\|_\LT}$
 & $\d\frac{|\bar{y}-\bar y_h|_\HO}{|\bar{y}|_\HO}$
&$\d\frac{\|\bar{y}-\bar y_h\|_\LT}{\|\bar{y}\|_\LT}$& Time \\[10pt]
\hline
\multicolumn{6}{|c|}{\rule{0pt}{2.5ex}$y_d=1$}\\[2pt]
\hline &&&&&\\[-11pt]
$10^{-2}$&1.65e-02&6.92e-04&1.17e-02&6.31e-04
&7.59e\splus00\\
\hline &&&&&\\[-11pt]
$10^{-4}$ &5.62e-02&4.64e-03&1.92e-02&9.29e-04
&8.20e\splus00\\
\hline &&&&&\\[-11pt]
$10^{-6}$&1.87e-01&3.99e-02&6.13e-01&4.51e-03
&2.53e\splus01\\
\hline
\multicolumn{6}{|c|}{\rule{0pt}{2.5ex}$y_d=x_1(1-x_1)x_2(1-x_2)$}\\[2pt]
\hline &&&&&\\[-11pt]
 $10^{-2}$&1.16e-02&2.65e-04&1.16e-02&4.26e-04
 &5.15e\splus00\\
 \hline &&&&&\\[-11pt]
 $10^{-4}$&1.47e-02&1.88e-04&1.17e-02&1.92e-04
 &6.12e\splus00\\
   \hline &&&&&\\[-11pt]
 $10^{-6}$&4.55e-02&8.82e-04&1.22e-02&1.86e-04
 &1.46e\splus01\\
 \hline
\end{tabular}
\par\bigskip
 \caption{Relative errors and solution times of the $P_1$ finite element method defined by
   \eqref{subeq:OriginalP1} for $y_d=1$ and
   $y_d=x_1(1-x_1)x_2(1-x_2)$,  with $h=2^{-6}$ and $\beta=10^{-2}, 10^{-2}, 10^{-6}$}
\label{table:P1Errors}
\end{table}
\begin{remark}\label{rem:OptimalControl}
  We can approximate the optimal control $\bar u$ in \eqref{eq:OCP} by
   $\bar u_h=-\beta^{-1}\bar p_h$.  It then follows from \eqref{eq:Lions2} that the relative error for
   $\bar u_h$ is identical to the relative error for $\bar p_h$.
\end{remark}
\section{Multigrid Algorithms}\label{sec:MG}
 Let $\cT_0$ be a triangulation of $\O$ and
 the triangulations $\cT_1, \cT_2, ...$  be generated from $\cT_0$ through a refinement process
 so that $h_k= h_{k-1}/2$ and the shape regularity of $\cT_k$ is inherited from the
 shape regularity of $\cT_0$.
 The $P_1$ finite element subspace of $H^1_0(\O)$ associated with $\cT_k$ is denoted by $V_k$.
\par
  We want to design multigrid methods for problems of the form
\begin{equation}\label{eq:kthProblem}
 \cB((p,y),(q,z))=F(q)+G(z)\qquad \forall\, (q,z)\in \FE{k},
\end{equation}
 where $F,G\in V_k'$.
\subsection{A Mesh-Dependent Inner Product}\label{subsec:MDIP}
 It is convenient to use a mesh-dependent inner product on
 $\FE{k}$ to rewrite \eqref{eq:kthProblem} in terms of an operator
 that maps $\FE{k}$ to $\FE{k}$.
 First we introduce a mesh-dependent inner product  on $V_k$:
\begin{equation}\label{eq:VkIP}
 (v,w)_k=h_k^2\sum_{x\in\cV_k}v(x)w(x) \qquad\forall\,v,w\in V_k,
\end{equation}
 where $\cV_k$ is the set of the interior vertices of $\cT_k$.
  We have
\begin{equation}\label{eq:LTEquivalence}
 (v,v)_k\approx\|v\|^2_\LT   \qquad \forall \, v\in V_k
\end{equation}
 by a standard scaling argument \cite{Ciarlet:1978:FEM,BScott:2008:FEM}, where the hidden constants
 only depend on the shape regularity of $\cT_0$.
\par
  We then define the mesh-dependent inner product $[\cdot,\cdot]_k$ on $\FE{k}$  by
\begin{equation}\label{eq:MDIP}
 [(p,y),(q,z)]_k=(p,q)_k+(y,z)_k.
\end{equation}
\par
 Let the operator $\fB_k:\FE{k}\longrightarrow \FE{k}$ be defined by
\begin{equation}\label{eq:fBk}
 [\fB_k(p,y),(q,z)]_k=\cB((p,y),(q,z)) \qquad \forall \, (p,y),(q,z) \in \FE{k}.
\end{equation}
 We can then rewrite \eqref{eq:kthProblem} in the form
\begin{equation}\label{eq:kthEq}
  \fB_k(p,y)=(f,g),
\end{equation}
 where $(f,g)\in \FE{k}$ is defined by
\begin{equation*}
  [(f,g),(q,z)]_k=F(q)+G(z) \qquad\forall\,(q,z)\in\FE{k}.
\end{equation*}
\par
 We take  the coarse-to-fine operator $I^k_{k-1}: \FE{k-1}\longrightarrow \FE{k}$ to be
   the natural injection and define the fine-to-coarse operator
    $I^{k-1}_k:\FE{k}\longrightarrow \FE{k-1}$ to be the transpose of $I^k_{k-1}$
     with respect to the mesh-dependent inner products,  i.e.,
\begin{equation}\label{eq:FTC}
 [I^{k-1}_k(p,y),(q,z)]_{k-1}=[(p,y),I^k_{k-1}(q,z)]_k \quad \forall\, (p,y)\in \FE{k}, (q,z)\in \FE{k-1}.
\end{equation}
%
\subsection{A Block-Diagonal Preconditioner}\label{subsec:BDP}
 Let $L_k:V_k\longrightarrow V_k$ be a linear  operator  symmetric with respect to
 the inner product $(\cdot,\cdot)_k$ on $V_k$ such that
\begin{equation}\label{eq:Lk}
  (L_k v,v)_k\approx \|v\|_\HB^2=\|v\|_\LT^2+\beta^\frac12|v|_\HO^2 \qquad \forall\,v\in V_k.
 \end{equation}
 Then the operator $\fC_k:\FE{k}\longrightarrow\FE{k}$ defined by
\begin{equation}\label{eq:fCk}
 \fC_k(p,y)=(L_k p, L_k y)
\end{equation}
\goodbreak\noindent
 is symmetric positive definite (SPD) with respect to $[\cdot,\cdot]_k$ and we have
\begin{equation}\label{eq:EnergyNormEquivalence}
 [\fC_k(p,y),(p,y)]_k\approx \|p\|^2_\HB+\|y\|^2_\HB \qquad  \forall \,(p,y)\in \FE{k},
\end{equation}
 where the hidden constants are independent of $k$ and $\beta$.
\begin{remark}\label{rem:Preconditioner}
  We will use $\fC_k^{-1}$  as a preconditioner in the constructions of the smoothing operators.
  In practice we can take $L_k^{-1}$ to be an approximate solve of the
  $P_1$ finite element discretization of the following boundary value problem:
\begin{equation}\label{eq:BVP}
 -\beta^{\frac{1}{2}}\Delta u+u=\phi \quad\text{in $\O$} \quad\text{and}\quad
 u=0 \quad \text{on $\p\O$}.
\end{equation}
 The  multigrid algorithms in Section~\ref{sec:MG} are $O(n)$ algorithms as long as
 $L_k^{-1}$ is also an $O(n)$ algorithm.
 We refer to \cite{MW:2011:SaddlePoint,ESW:2014:FEFIS}
 for
 the general construction of block diagonal preconditioners for
 saddle point  problems arising
 from the discretization of partial differential equations.
\end{remark}
\begin{lemma}\label{lem:Fundamental}
 We have
\begin{equation}\label{eq:Fundamental}
 [\FO(p,y),(p,y)]_k\approx\|p\|^2_\HB+\|y\|^2_\HB \qquad \forall \, (p,y)\in \FE{k},
\end{equation}
 where the hidden constants are independent of $k$ and $\beta$.
\end{lemma}
\begin{proof}
  Let $(p,y)\in \FE{k}$ be arbitrary and $(r,x)=\fC^{-1}_kB_k(p,y)$. It follows from
  \eqref{eq:DiscreteInfSup}, \eqref{eq:fBk}, \eqref{eq:EnergyNormEquivalence} and duality
  that
\begin{align*}
 [\FO(p,y),(p,y)]_k&=[\fC_k(\fC^{-1}_k\fB_k)(p,y),\fC^{-1}_k\fB_k(p,y)]_k\\
 &=[\mathfrak{C}_k(r,x),(r,x)]_k\\
 &=\sup_{(q,z)\in \FE{k}}\frac{[\fC_k(r,x),(q,z)]^2_k}{[\fC_k(q,z),(q,z)]_k}\\
 &\approx \sup_{(q,z)\in \FE{k}}\ \frac{[\fB_k(p,y),(q,z)]_k^2}{\|q\|^2_\HB
 +\|z\|^2_\HB}\\
 &=\sup_{(q,z)\in V_k\times V_k}
 \frac{\mathcal{B}((p,y),(q,z))^2}{\|q\|^2_{H^1_{\beta}(\Omega)}+\|z\|^2_{H^1_{\beta}(\Omega)}}
 \approx\|p\|^2_\HB+\|y\|^2_\HB.
\end{align*}
\end{proof}
\begin{lemma}\label{lem:SpectralBounds}
  The minimum and maximum eigenvalues of
  $\fB_k\fC_k^{-1}\fB_k$ satisfy the following bounds$:$
\begin{align}
 \lambda_{\min}(\FO)&\geq C_{\min}, \label{eq:Min}\\
 \lambda_{\max}(\FO)&\leq C_{\max}(1+\beta^{\frac{1}{2}}h_k^{-2}),
  \label{eq:Max}
\end{align}
 where the positive constants $C_{\min}$ and $C_{\max}$ are independent of $k$ and $\beta$.
\end{lemma}
\begin{proof} We have, from \eqref{eq:LTEquivalence} and \eqref{eq:MDIP},
\begin{equation}\label{eq:Cond1}
  [(p,y),(p,y)]_k\approx \|p\|_\LT^2+\|y\|_\LT^2 \qquad\forall\,(p,y)\in\FE{k},
\end{equation}
 where the hidden constants only depend on the shape regularity of $\cT_0$.  It follows from
 \eqref{eq:HB},
 \eqref{eq:Fundamental} and \eqref{eq:Cond1} that
\begin{equation}\label{eq:Cond2}
  [\FO (p,y),(p,y)]_k\geq C_{\min}[(p,y),(p,y)]_k \qquad\forall\,(p,y)\in\FE{k},
\end{equation}
 which then implies \eqref{eq:Min} by the Rayleigh quotient formula.
\par
  By a standard inverse estimate \cite{Ciarlet:1978:FEM,BScott:2008:FEM}, we have
\begin{equation*}
  \|v\|_\HB^2=\|v\|_\LT^2+\beta^\frac12|v|_\HO^2\leq (1+C\beta^\frac12 h_k^{-2})\|v\|_\LT^2
  \qquad\forall\,v\in V_k,
\end{equation*}
 where the positive constant $C$ depends only on the shape regularity of $\cT_0$.   It then follows from \eqref{eq:HB}, \eqref{eq:Fundamental} and \eqref{eq:Cond1} that
\begin{equation*}
  [\FO(p,y),(p,y)]_k\leq C_{\max}(1+\beta^{\frac{1}{2}}h_k^{-2})[(p,y),(p,y)]_k \qquad\forall\,(p,y)\in\FE{k},
\end{equation*}
 and hence \eqref{eq:Max} holds because of the Rayleigh quotient formula.
\end{proof}
\begin{remark}\label{rem:Conditioning}
 It follows from \eqref{eq:Min} and \eqref{eq:Max} that the operator
 $\FO$ is well-conditioned when $\beta^\frac12 h_k^{-2}=O(1)$.
\end{remark}
\subsection{A $\bm{W}$-Cycle Multigrid Algorithm}\label{subsec:WCycle}
 Let the output of the W-cycle algorithm for \eqref{eq:kthEq} with initial guess $(p_0, y_0)$
 and $m_1$ (resp., $m_2$) pre-smoothing (resp., post-smoothing) steps be denoted by
 $MG_W(k,(f,g), (p_0, y_0),m_1,m_2)$.
\par
 We use a direct solve for $k=0$, i.e., we take  $MG_W(0,(f,g), (p_0, y_0),m_1,m_2)$ to be
 $B_0^{-1}(f,g)$. For $k\geq 1$, we compute  $MG_W(k,(f,g), (p_0, y_0),m_1,m_2)$ in three steps.
\par\medskip\noindent
{\em Pre-Smoothing}\quad
 We compute $(p_1,y_1), \ldots, (p_{m_1},y_{m_1})$
 recursively by
\begin{equation}\label{eq:PreSmoothing}
 (p_j,y_j)=(p_{j-1},y_{j-1})+\lambda_k\fC_k^{-1}\fB_k((f,g)-\fB_k(p_{j-1},y_{j-1}))
\end{equation}
 for $1\leq j\leq m_1$.  The choice of the damping factor $\lambda_k$ will be given below
 in \eqref{eq:Damping1} and \eqref{eq:Damping2}.
\par\medskip\noindent
{\em Coarse Grid Correction}\quad
 Let $(f',g')=I^{k-1}_k((f,g)-B_k(p_{m_1}, y_{m_1}))$ be the transferred
  residual of $(p_{m_1},y_{m_1})$ and let $(p'_1,y'_1),(p'_2,y'_2)\in \FE{k-1}$ be computed  by
\begin{subequations}\label{eq:CoarseGrid}
\begin{align}
 (p'_1,y'_1)&=MG_W(k-1,(f',g'),(0,0),m_1,m_2),\\
 (p'_2,y'_2)&=MG_W(k-1,(f',g'),(p'_1,y'_1),m_1,m_2).
\end{align}
\end{subequations}
 We then take $(p_{m_1+1},y_{m_1+1})$ to be $(p_{m_1},y_{m_1})+I^k_{k-1}(p'_2,y'_2)$.

\par\medskip\noindent
{\em Post-Smoothing}\quad
 We compute $(p_{m_1+2},y_{m_1+2})$,
 $\!\ldots\,$,$\,(p_{m_1+m_2+1},y_{m_1+m_2+1})$
  recursively by
\begin{equation}\label{eq:PostSmoothing}
 (p_j,y_j)=(p_{j-1},y_{j-1})+\lambda_k\fB_k\fC_k^{-1}((f,g)-\fB_k(p_{j-1},y_{j-1}))
\end{equation}
 for  $m_1+2\leq j \leq m_1+m_2+1$.
\par\medskip\noindent
 The final output is $MG_W(k,(f,g), (p_0, y_0),m_1,m_2)=(p_{m_1+m_2+1},y_{m_1+m_2+1})$.
\par\medskip
 To complete the description of the algorithm, we choose the damping factor $\lambda_k$ as follows:
\begin{alignat}{3}
   \lambda_k&=\frac{2}{\lambda_{\min}(\FO)+\lambda_{\max}(\FO)}
     &\qquad&\text{if $\beta^\frac12 h_k^{-2}< 1$},\label{eq:Damping1}\\
 \intertext{and}
   \lambda_k&= [C_\dag({1+\beta^\frac12h_k^{-2}})]^{-1} &\qquad&
   \text{if $\beta^\frac12 h_k^{-2}\geq 1$},
     \label{eq:Damping2}
\end{alignat}
 where $C_\dag$ is greater than or equal to the constant $C_{\max}$ in \eqref{eq:Max}.
\begin{remark}\label{rem:Richardson}
  Note that the post-smoothing step is exactly the Richardson iteration for the equation
 \begin{equation*}
   \FO(p,y)=\fB_k\fC_k^{-1}(f,g),
 \end{equation*}
  which is equivalent to \eqref{eq:kthEq}.
\end{remark}
\begin{remark}\label{rem:Damping}
  In the case where $\beta^\frac12 h_k^{-2}< 1$,
  the choice of $\lambda_k$ is motivated by the well-conditioning of $\FO$
 (cf. Remark~\ref{rem:Conditioning}) and the optimal choice of damping factor
  for the Richardson iteration \cite[p.~114]{Saad:2003:IM}.  In practice the relation
  \eqref{eq:Damping1} only holds approximately, but it affects neither the analysis
  nor the performance of
  the $W$-cycle algorithm.
  In the case where
  $\beta^\frac12 h_k^{-2}\geq 1$, the choice of $\lambda_k$ is motivated by the
  condition $\lambda_{\max}(\lambda_k\FO)\leq 1$ (cf. \eqref{eq:Max}) that
  will ensure the highly oscillatory part of the error is damped out when Richardson
  iteration is used as a smoother for an ill conditioned system
  (cf. Lemma~\ref{lem:PostSmoothing2}).
\end{remark}
\subsection{A $\bm{V}$-Cycle Multigrid Algorithm}\label{subsec:VCycle}
 Let the output of the V-cycle algorithm for \eqref{eq:kthEq} with initial guess
  $(p_0, y_0)$  and $m_1$ (resp., $m_2$) pre-smoothing (resp., post-smoothing)
   steps be denoted by $MG_V(k,(f,g), (p_0, y_0),m_1,m_2)$. The difference between
   the  computations of $MG_V(k,(f,g), (p_0, y_0),m_1,m_2)$ and
   $MG_W(k,(f,g), (p_0, y_0),m_1,m_2)$ is only
   in the coarse grid correction step, where we compute
\begin{equation*}
 (p'_1,y'_1)=MG_V(k-1,(f',g'),(0,0),m_1,m_2)
\end{equation*}
 and take $(p_{m_1+1},y_{m_1+1})$ to be $(p_{m_1},y_{m_1})+I^k_{k-1}(p'_1,y'_1)$.
\begin{remark}\label{rem:VCycle}
   We will focus on the analysis of the $W$-cycle algorithm in this paper.
   But numerical results indicate that the performance of the $V$-cycle algorithm is also robust
  respect to $k$ and $\beta$.
\end{remark}
\section{Smoothing and Approximation Properties}\label{sec:SA}
 We will develop in this section two key ingredients for the convergence analysis of the
 $W$-cycle algorithm, namely, the smoothing and approximation properties.
 They will be expressed in terms of a scale of mesh-dependent norms defined by
\begin{equation}\label{eq:MDNorms}
 \trinorm{(p,y)}_{s,k}=[(\FO)^s(p,y),(p,y)]^\frac12_k \qquad \forall (p,y) \in \FE{k}.
\end{equation}
\par
 Note that
\begin{alignat}{3}
  \trinorm{(p,y)}_{0,k}^2&\approx \|p\|_\LT^2+\|y\|_\LT^2  &\qquad& \forall\,(p,y)\in\FE{k}
  \label{eq:0kNorm}\\
 \intertext{by \eqref{eq:Cond1}, and}
 \trinorm{(p,y)}_{1,k}^2&\approx \|p\|_\HB^2+\|y\|_\HB^2 &\qquad&\forall\,(p,y)\in\FE{k}
 \label{eq:1kNorm}
\end{alignat}
 by \eqref{eq:Fundamental}.
%
\subsection{Post-Smoothing Properties}\label{subsec:PS}
 The error propagation operator for one post-smoothing step defined by
  \eqref{eq:PostSmoothing} is given by
\begin{equation}\label{eq:Rk}
  R_k=Id_k-\lambda_k\FO,
\end{equation}
 where $Id_k$ is the identity operator on $\FE{k}$.
\begin{lemma}\label{lem:PostSmoothing1}
 In the case where $\beta^\frac12 h_k^{-2}< 1$, we have
\begin{equation}\label{eq:PostSmoothing1}
 \trinorm{R_k(p,y)}_{1,k}\leq \gamma \trinorm{(p,y)}_{1,k} \qquad\forall\,(p,y)\in \FE{k},
\end{equation}
 where the constant $\gamma\in (0,1)$ is independent of $k$ and $\beta$.
\end{lemma}
\begin{proof}  In this case $\lambda_k$ given by \eqref{eq:Damping1} is the optimal damping parameter
 for the Richardson iteration and we have
   $$C_{\min}\leq \lambda_{\min}(\FO)\leq \lambda_{\max}(\FO))< 2C_{\max}$$
 by Lemma~\ref{lem:SpectralBounds}.
 It follows that (cf. \cite[p.~114]{Saad:2003:IM})
\begin{align*}
  \trinorm{R_k(p,y)}_{1,k}&=[\FO R_k(p,y),R_k(p,y)]_k^\frac12\\
  &\leq\Big(\frac{\lambda_{\max}(\FO)-\lambda_{\min}(\FO)}{\lambda_{\max}(\FO)+\lambda_{\min}(\FO)}\Big)
    [(p,y),(p,y)]_k^\frac12\\
    &\leq \Big(\frac{2C_{\max}-C_{\min}}{2C_{\max}+C_{\min}}\Big)\trinorm{(p,y)}_{1,k}.
     \end{align*}
 Therefore \eqref{eq:PostSmoothing1} holds for
 $\gamma=(2C_{\max}-C_{\min})/(2C_{\max}+C_{\min})$.
\end{proof}
\begin{lemma}\label{lem:PostSmoothing2}
  In the case where $\beta^\frac12 h_k^{-2}\geq 1$, we have, for $0\leq s\leq 1$,
\begin{equation*}
 \trinorm{R_k^m(p,y)}_{1,k}\leq C(1+\beta^{\frac{1}{2}}h_k^{-2})^{s/2}m^{-s/2}
  \trinorm{(p,y)}_{1-s,k} \quad \forall (p,y)\in \FE{k},
\end{equation*}
 where the positive constant $C$ is independent of $k$ and $\beta$.
\end{lemma}
\begin{proof}  In this case $\lambda_k$ is given by \eqref{eq:Damping2} and
 $\lambda_{\max}(\lambda_k\FO)\leq 1$.  It follows from
 \eqref{eq:Damping2},  \eqref{eq:MDNorms},
 \eqref{eq:Rk}, calculus and the spectral theorem that
\begin{align*}
 \trinorm{R^m_k(p,y)}_{1,k}^2&=[\FO R_k^m(p,y), R_k^m(p,y)]_k\\
 &=\lambda_k^{-s}[(\FO)^{1-s}(\lambda_k \FO)^sR_k^m(p,y),R_k^m(p,y)]_k\\
 &\leq C_\dag^{-s}(1+\beta^{\frac{1}{2}}h_k^{-2})^s
 \max_{0\leq x\leq 1}[(1-x)^{2m}x^s][(\FO)^{1-s}(p,y),(p,y)]_k\\
 &\leq C(1+\beta^{\frac{1}{2}}h_k^{-2})^s m^{-s}\trinorm{(p,y)}_{1-s,k}^2.
\end{align*}
\end{proof}
\begin{remark}\label{rem:Rk}
  In the special case where $s=0$, the calculation in the proof of Lemma~\ref{lem:PostSmoothing2}
  shows that
\begin{equation*}
   \trinorm{R_k(p,y)}_{1,k}\leq \trinorm{(p,y)}_{1,k} \qquad\forall\,(p,y)\in \FE{k}.
\end{equation*}
\end{remark}
\subsection{An Approximation Property}\label{subsec:AP}
 We define the Ritz projection operator $P^{k-1}_k:\FE{k}\rightarrow\FE{k-1}$ to be the transpose of the
 coarse-to-fine operator $I^k_{k-1}:\FE{k-1}\rightarrow \FE{k}$ with
  respect to the variational form $\cB(\cdot,\cdot)$.
    Recall that $I_{k-1}^k$ is the natural injection.  Therefore we have,
 for any $(p,y)\in\, \FE{k}$ and $(q,z)\in\,\FE{k-1}$,
\begin{equation}\label{eq:Ritz}
 \cB(P^{k-1}_k(p,y),(q,z))=\cB((p,y),I_{k-1}^k(q,z))=\cB((p,y),(q,z)).
\end{equation}
 It follows that
\begin{equation*}
  P_k^{k-1}I_{k-1}^k=Id_{k-1}
\end{equation*}
 and hence
\begin{equation}\label{eq:Projections}
  (I_{k-1}^kP_k^{k-1})^2=I_{k-1}^kP_k^{k-1} \quad\text{and}\quad
  (Id_k-I_{k-1}^kP_k^{k-1})^2=Id_k-I_{k-1}^kP_k^{k-1}.
\end{equation}
 Moreover we have the following Galerkin orthogonality:
\begin{equation}\label{eq:Galerkin}
  \cB((\AP)(p,y),I_{k-1}^k (q,z))=0 \quad\forall\,(p,y)\in\FE{k},(q,z)\in\FE{k-1}.
\end{equation}
\par
 The effect of the operator $\AP$ is measured by the following
  approximation property.
\begin{lemma}\label{lem:Approximation}
  There exists a positive constant $C$ independent of $k$ and $\beta$ such that
\begin{equation*}
 \trinorm{(\AP)(p,y)}_{0,k}\leq C (1+{\beta}^{\frac{1}{2}}h_k^{-2})^{\frac{1}{2}}
  \beta^{-\frac{1}{2}}h_k^2\trinorm{(p,y)}_{1,k} \quad\forall\,(p,y)\in \FE{k}.
\end{equation*}
%
%
\end{lemma}
\begin{proof}
 Let $(p,y)\in \FE{k}$ be arbitrary and
\begin{equation}\label{eq:Approximation0}
  (\zeta,\mu)=(\AP)(p,y).
\end{equation}
 In view of \eqref{eq:0kNorm},  it suffices to establish the estimate
\begin{equation}\label{eq:Approximation1}
 \|\zeta\|_\LT+\|\mu\|_\LT
  \lesssim (1+{\beta}^{\frac{1}{2}}h_k^{-2})^{\frac{1}{2}}\beta^{-\frac{1}{2}}h_k^2
  \trinorm{(p,y)}_{1,k}
\end{equation}
 by a duality argument.
\par
 Let  $(\xi,\theta)\in\ES$ be defined by
\begin{equation}\label{eq:Approximation2}
 \cB((\xi,\theta),(q,z))=(\zeta,q)_\LT+(\mu,z)_\LT \qquad \forall\, (q,z)\in \ES,
\end{equation}
 and
  $(\xi_{k-1},\theta_{k-1})\in \FE{k-1}$ be defined by
\begin{equation}\label{eq:Approximation3}
  \cB((\xi_{k-1},\theta_{k-1}),(q,z))=(\zeta,q)_\LT+(\mu,z)_\LT \qquad
  \forall \,(q,z)\in \FE{k-1}.
\end{equation}
 Since $h_{k-1}=2h_k$,  we have, according to Lemma~\ref{lem:ConcreteErrors},
\begin{align}\label{eq:Approximation4}
 \|\xi-\xi_{k-1}\|_\HB+\|\theta-\theta_{k-1}\|_\HB\lesssim
 (1+{\beta}^{\frac{1}{2}}h_k^{-2})^{\frac{1}{2}}\beta^{-\frac{1}{2}}h_k^2(\|\zeta\|_\LT
 +\|\mu\|_\LT).
\end{align}
\par
 Putting \eqref{eq:BUpper}, \eqref{eq:1kNorm}, \eqref{eq:Galerkin},
  \eqref{eq:Approximation0},
 \eqref{eq:Approximation2} and \eqref{eq:Approximation3}
 together, we find
\begin{align*}
  \|\zeta\|_\LT^2+\|\mu\|_\LT^2 &= \cB((\xi,\theta),(\zeta,\mu))\\
 &= \cB((\xi,\theta),(\AP)(p,y))\\
 &= \cB((\xi,\theta)-(\xi_{k-1},\theta_{k-1}),(Id_k-I^k_{k-1}P^{k-1}_k)(p,y))\\
 &= \mathcal{B}((\xi,\theta)-(\xi_{k-1},\theta_{k-1}),(p,y))\\
 &\lesssim
    (\|\xi-\xi_{k-1}\|^2_\HB+\|\theta-\theta_{k-1}\|^2_\HB)^\frac12
    (\|p\|^2_\HB+\|y\|^2_\HB)^\frac12\\
 &\lesssim (1+{\beta}^{\frac{1}{2}}h_k^{-2})^{\frac{1}{2}}\beta^{-\frac{1}{2}}h_k^2(\|\zeta\|_\LT
 +\|\mu\|_\LT)\trinorm{(p,y)}_{1,k},
\end{align*}
 which implies \eqref{eq:Approximation1}.
\end{proof}
\goodbreak
\par
 We will also need the following stability estimates.
\begin{lemma}\label{lem:Stability}
 We have
\begin{alignat}{3}
 \trinorm{I^k_{k-1}(q,z)}_{1,k}&\approx\trinorm{(q,z)}_{1,k-1}
   & \qquad &\forall\, (q,z)\in \FE{k-1},\label{eq:C2FStability}\\
 \trinorm{P^{k-1}_k(p,y)}_{1,k-1}&\lesssim \trinorm{(p,y)}_{1,k}
  &\qquad&\forall\, (p,y)\in \FE{k},\label{eq:RitzStability}
\end{alignat}
 where the hidden constants are independent of $k$ and $\beta$.
\end{lemma}
\begin{proof} The estimate \eqref{eq:C2FStability} follows from
 \eqref{eq:1kNorm} and the fact that $I_{k-1}^k$ is the natural injection.
 The estimate \eqref{eq:RitzStability} then follows from \eqref{eq:DiscreteInfSup},
 \eqref{eq:1kNorm},
 \eqref{eq:Ritz} and \eqref{eq:C2FStability} :
\begin{align*}
  \trinorm{P_k^{k-1}(p,y)}_{1,k-1}&\approx\sup_{(q,z)\in\FE{k-1}}
  \frac{\cB\big(P_k^{k-1}(p,y),(q,z)\big)}{\trinorm{(q,z)}_{1,k-1}}\\
  &=\sup_{(q,z)\in\FE{k-1}}
  \frac{\cB\big((p,y),I_{k-1}^k(q,z)\big)}{\trinorm{(q,z)}_{1,k-1}}
  \lesssim \trinorm{(p,y)}_{1,k}.
\end{align*}
\end{proof}
%
\section{Convergence Analysis of the $W$-Cycle Algorithm}\label{sec:Convergence}
 Let $E_k:\FE{k}\longrightarrow \FE{k}$ be the error propagation operator for the
 $k$-th level $W$-cycle algorithm.  We have the following well-known recursive relation
 (cf. \cite{Hackbusch:1985:MMA,MMB:1987:VMT,BZ:2000:AMG}):
\begin{equation}\label{eq:Recursion}
 E_k=R_k^{m_2}(Id_k-I^k_{k-1}P^{k-1}_k+I^k_{k-1}E^2_{k-1}P^{k-1}_k)S_k^{m_1},
\end{equation}
 where
\begin{equation}\label{eq:Sk}
 S_k=Id_k-\lambda_k\mathfrak{C}^{-1}_k\fB_k^2
\end{equation}
 is the error propagation operator for one pre-smoothing step (cf.  \eqref{eq:PreSmoothing}).
\par
 Note that $S_k$ is the transpose of $R_k$ with respect to the
 variational form $\cB(\cdot,\cdot)$ by \eqref{eq:fBk} and \eqref{eq:Rk}:
\begin{align}\label{eq:Adjoint}
 \cB(S_k(p,y),(q,z))&=[\fB_k(Id_k-\lambda_k\mathfrak{C}^{-1}_k\fB_k^2)(p,y),(q,z)]_k
 \notag\\
   &=[\fB_k(p,y),(Id_k-\lambda_k\FO)(q,z)]_k\\
   &=\cB((p,y),R_k(q,z))&\forall\,(p,y),(q,z)\in \FE{k}.\notag
\end{align}
\par
 The relations \eqref{eq:Ritz} and \eqref{eq:Adjoint} lead to the following useful
 result.
\begin{lemma}\label{lem:Useful}
 We have
\begin{equation*}
  \|R_k^m(\AP)\|\approx \|(\AP)S_k^m\|,
\end{equation*}
 where $\|\cdot\|$ denotes the operator norm with respect to $\trinorm{\cdot}_{1,k}$
 and the hidden constants are independent of $k$ and $\beta$.
\end{lemma}
\begin{proof}  It follows from  \eqref{eq:DiscreteInfSup},
 \eqref{eq:1kNorm},  \eqref{eq:Ritz} and \eqref{eq:Adjoint} that
\begin{align*}
 &\trinorm{(\AP)S_k^{m}(p,y)}_{1,k}\notag\\
 &\hspace{40pt}\approx\sup_{(q,z)\in\FE{k}}
 \frac{\cB\big((\AP)S_k^{m}(p,y),(q,z)\big)}{\trinorm{(q,z)}_{1,k}}\\
 &\hspace{40pt}=\sup_{(q,z)\in\FE{k}}
 \frac{\cB\big((p,y),R_k^{m}(\AP)(q,z)\big)}{\trinorm{(q,z)}_{1,k}}\notag\\
 &\hspace{40pt}\lesssim \trinorm{(p,y)}_{1,k} \|R_k^m(\AP)\|&\forall\,(p,y)\in\FE{k},\notag
\end{align*}
 and hence
\begin{equation*}
  \|(\AP)S_k^m\|\lesssim \|R_k^m(\AP)\|.
\end{equation*}
\par
 The  estimate in the other direction is established by a similar argument.
\end{proof}
%
\subsection{Convergence of the Two-Grid Algorithm}\label{subsec:TwoGrid}
 In the two-grid algorithm the coarse grid residual equation is solved exactly.
 By setting $E_{k-1}=0$ in \eqref{eq:Recursion},  we obtain
 the error propagation operator $R_k^{m_2}(\AP)S_k^{m_1}$  for the
 two-grid algorithm with $m_1$ (resp., $m_2$) pre-smoothing (resp., post-smoothing) steps.
\par
 We will separate the convergence analysis into two cases.
%
\par\smallskip\noindent{\bf The case where $\beta^\frac12 h_k^{-2}< 1$.}
\quad
 Here we can apply Lemma~\ref{lem:PostSmoothing1} which states that
 $R_k$ is a contraction with respect to $\trinorm{\cdot}_{1,k}$
 and the  contraction number $\gamma$ is independent of $k$ and $\beta$.
\begin{lemma}\label{lem:TG1}
  In the case where $\beta^\frac12 h_k^{-2}< 1$, there exists a positive constant $C_\sharp$
  independent of $k$ and $\beta$ such that
\begin{equation}\label{eq:TG1}
   \|R_k^{m_2}(\AP)S_k^{m_1}\|\leq C_\sharp
   \gamma^{m_1+m_2},
\end{equation}
 where $\|\cdot\|$ is the operator norm with respect to $\trinorm{\cdot}_{1,k}$.
\end{lemma}
\begin{proof}  We have, from Lemma~\ref{lem:PostSmoothing1}
 and Lemma~\ref{lem:Stability},
\begin{align*}
& \trinorm{R_k^{m}(\AP)(p,y)}_{1,k}\\
 &\hspace{40pt}\leq \gamma^{m}\trinorm{(\AP)(p,y)}_{1,k}
 \lesssim \gamma^{m}\trinorm{(p,y)}_{1,k} \qquad\forall\,(p,y)\in\FE{k},
 \notag
\end{align*}
 and hence
\begin{equation}\label{eq:TG11}
  \| R_k^m(\AP)\|\lesssim \gamma^m.
\end{equation}
\par
 It then follows from Lemma~\ref{lem:Useful} that
\begin{equation}\label{eq:TG12}
  \| (\AP)S_k^m\|\lesssim \gamma^m.
\end{equation}
\par
 Finally we establish \eqref{eq:TG1} by
 combining \eqref{eq:Projections}, \eqref{eq:TG11} and  \eqref{eq:TG12}:
\begin{align*}
  &\|R_k^{m_2}(\AP)S_k^{m_1}\|\\
  &\hspace{30pt}= \|R_k^{m_2}(\AP)(\AP)S_k^{m_1}\|\\
  &\hspace{30pt}\leq \|R_k^{m_2}(\AP)\|\|(\AP)S_k^{m_1}\|
  \lesssim \gamma^{m_1+m_2}.
\end{align*}
\end{proof}
%
\par\smallskip\noindent{\bf The case where $\beta^\frac12 h_k^{-2}\geq 1$.}
\quad
 Here we can apply Lemma~\ref{lem:PostSmoothing2}.
\begin{lemma}\label{lem:TG2}
 In the case where $\beta^\frac12 h_k^{-2}\geq  1$, there exists a positive constant
 $C_\flat$ independent of $k$ and $\beta$ such that
\begin{equation}\label{eq:TG2}
   \|R_k^{m_2}(\AP)S_k^{m_1}\|
         \leq C_\flat[\max(1,m_1)\max(1,m_2)]^{-\frac12},
\end{equation}
 where $\|\cdot\|$ is the operator norm with respect to $\trinorm{\cdot}_{1,k}$.
\end{lemma}
\begin{proof}  Let $m$ be any positive integer.
 We have, from Lemma~\ref{lem:PostSmoothing2}
 and Lemma~\ref{lem:Approximation},
\begin{align*}
  &\trinorm{R_k^m(\AP)(p,y)}_{1,k} \\
  &\hspace{40pt}\lesssim (1+\beta^\frac12 h_k^{-2})^\frac12 m^{-\frac12}
   \trinorm{(\AP)(p,y)}_{0,k}\\
   &\hspace{40pt}\lesssim (1+\beta^\frac12 h_k^{-2})^\frac12 m^{-\frac12}
   (1+{\beta}^{\frac{1}{2}}h_k^{-2})^{\frac{1}{2}}
  \beta^{-\frac{1}{2}}h_k^2\trinorm{(p,y)}_{1,k} \\
  &\hspace{40pt}= m^{-\frac12}(\beta^{-\frac12}h_k^2+1)
  \trinorm{(p,y)}_{1,k}\\
  &\hspace{40pt}\leq 2m^{-\frac12}
  \trinorm{(p,y)}_{1,k}  &\forall\,(p,y)\in\FE{k},
\end{align*}
 and hence
\begin{equation}\label{eq:TG21}
 \|R_k^m(\AP)\|\lesssim m^{-\frac12}.
\end{equation}
\par
 It then follows from Lemma~\ref{lem:Useful} that
\begin{equation}\label{eq:TG22}
 \|(\AP)S_k^m\|\lesssim m^{-\frac12}.
\end{equation}
 Combining \eqref{eq:Projections}, \eqref{eq:TG21} and \eqref{eq:TG22},
 we obtain for $m_1,m_2\geq1$,
\begin{align*}
 &\|R_k^{m_2}(\AP)S_k^{m_1}\|=\|R_k^{m_2}(\AP)(\AP)S_k^{m_1}\|\\
    &\hspace{70pt}\leq \|R_k^{m_2}(\AP)\|\|(\AP)S_k^{m_1}\|\\
    &\hspace{70pt}\lesssim (m_1m_2)^{-\frac12}.
\end{align*}
\par
 The cases where $m_1=0$ or $m_2=0$ follow directly from \eqref{eq:C2FStability},
 \eqref{eq:RitzStability}, \eqref{eq:TG21} and \eqref{eq:TG22}.
\end{proof}
%
\subsection{Convergence of the $\bm W$-Cycle Algorithm}\label{subsec:WCycleConvergence}
 We will derive error estimates for the $W$-cycle algorithm through \eqref{eq:Recursion} and
 the results for the two-grid algorithm in Section~\ref{subsec:TwoGrid}.  For simplicity
 we will focus on the symmetric $W$-cycle algorithm where $m_1=m_2=m\geq 1$.
\par
 According to Remark~\ref{rem:Rk}, we have
\begin{equation}\label{eq:RkTrivial}
  \|R_k^m\|\leq 1  \qquad\text{for}\;m\geq 1.
\end{equation}
 It follows from \eqref{eq:DiscreteInfSup}, \eqref{eq:1kNorm}, \eqref{eq:Adjoint} and
 \eqref{eq:RkTrivial} that
\begin{align*}
   \trinorm{S_k^m(p,y)}_{1,k}&\approx
   \sup_{(q,z)\in\FE{k}}\frac{\cB\big(S_k^m(p,y),(q,z)\big)}
   {\trinorm{(q,z)}_{1,k}}\\
   &= \sup_{(q,z)\in\FE{k}}\frac{\cB\big((p,y),R_k^m(q,z)\big)}
   {\trinorm{(q,z)}_{1,k}}
   \lesssim \trinorm{(p,y)}_{1,k} \qquad \forall\,(p,y)\in\FE{k},
\end{align*}
 which means there exists a positive constant $C_S$ independent of $k$ and $\beta$ such that
\begin{equation}\label{eq:SkTrivial}
  \|S_k^m\|\leq C_S \qquad\text{for}\;m\geq 1.
\end{equation}
\par
 Putting Lemma~\ref{lem:Stability}, \eqref{eq:Recursion}, \eqref{eq:RkTrivial} and
 \eqref{eq:SkTrivial} together, we obtain the recursive estimate
\begin{equation}\label{eq:RecursiveEstimate}
 \|E_k\|\leq \|R_k^m(\AP)S_k^m\|+C_*\|E_{k-1}\|^2
 \qquad\text{for}\;k\geq 1,
\end{equation}
 where the positive constant $C_*$ is independent of $k$ and $\beta$.
 The behavior of $\|E_k\|$ is therefore determined by \eqref{eq:RecursiveEstimate},
 the behavior of $\|R_k^m(\AP)S_k^m\|$, and
 the initial condition
\begin{equation}\label{eq:IC}
 \|E_0\|=0.
\end{equation}
\par
 Specifically, for $\beta^\frac12 h_k^{-2}< 1$, we have
\begin{equation}\label{eq:CoarseLevels}
 \|E_k\|\leq C_\sharp \gamma^{2m}+C_*\|E_{k-1}\|^2
\end{equation}
 by Lemma~\ref{lem:TG1}, and for $\beta^\frac12 h_k^{-2}\geq 1$, we have
\begin{equation}\label{eq:FineLevels}
  \|E_k\|\leq C_\flat m^{-1}+C_*\|E_{k-1}\|^2
\end{equation}
 by Lemma~\ref{lem:TG2}.
\par
 The following result is useful for the analysis of \eqref{eq:IC}--\eqref{eq:FineLevels}.
\begin{lemma}\label{lem:AbstractBdd}
  Let $\alpha_k$ $(k=0,1,2,\ldots)$ be a sequence of nonnegative numbers such that
\begin{equation}\label{eq:AbstractRE}
  \alpha_k\leq 1+\delta \alpha_{k-1}^2 \qquad\text{for}\quad k\geq1,
\end{equation}
 where the positive constant $\delta$ satisfies
\begin{equation}\label{eq:delta}
  \delta\leq \frac{1}{4(1+\alpha_0)}.
\end{equation}
 Then we have
\begin{equation}\label{eq:AbstractBdd}
  \alpha_k\leq 2+ 4^{1-2^k}\alpha_0 \qquad\text{for}\quad k\geq 0.
\end{equation}
\end{lemma}
\begin{proof}  The bound \eqref{eq:AbstractBdd} holds trivially for $k=0$.
  Suppose it holds for $k\geq0$.
  We have, by \eqref{eq:AbstractRE} and \eqref{eq:delta},
\begin{align*}
   \alpha_{k+1}\leq 1+\delta \alpha_{k}^2
         &\leq 1+\delta(2+4^{1-2^k}\alpha_0)^2\\
         &=1+\delta(4+4^{1-2^k}4\alpha_0)+(\delta\alpha_0) 4^{2-2^{k+1}}\alpha_0\\
         &\leq 1+\delta(4+4\alpha_0)+\Big(\frac14\Big)4^{2-2^{k+1}}\alpha_0
         \leq 2+4^{1-2^{k+1}}\alpha_0.
\end{align*}
 Therefore the bound \eqref{eq:AbstractBdd} holds for $k\geq0$ by mathematical induction.
\end{proof}
\begin{theorem}\label{thm:WCycle}
  Let $k_*$ be the largest positive integer such that $\beta^\frac12 h_{k_*}^{-2}< 1$.
  There exists a positive integer $m_*$ independent of $k$ such that
  $m\geq m_*$ implies
\begin{alignat}{3}
  \|E_k\|&\leq 2C_\sharp \gamma^{2m}   &\qquad&\forall\,1\leq k\leq k_*,  \label{eq:CoarseEst}\\
  \|E_k\|&\leq 2C_\flat m^{-1}+ 4^{1-2^{k-k_*}}(2C_\sharp \gamma^{2m})
     &\qquad&\forall\,k\geq k_*+1,\label{eq:FineEst}
\end{alignat}
 where $\|\cdot\|$ is the operator norm with respect to $\trinorm{\cdot}_{1,k}$.
\end{theorem}
\begin{proof}  For $1\leq k\leq k_*$, we take
   $\alpha_k=\|E_k\|/(C_\sharp \gamma^{2m})$
 and observe that
\begin{equation*}
  \alpha_k\leq 1+(C_*C_\sharp\gamma^{2m})\alpha_{k-1}^2
\end{equation*}
 by \eqref{eq:CoarseLevels}.  It then follows from \eqref{eq:IC} and Lemma~\ref{lem:AbstractBdd}
 that $\alpha_k\leq 2$, or equivalently
\begin{equation*}
  \|E_k\|\leq 2 C_\sharp\gamma^{2m},
\end{equation*}
 provided that
\begin{equation}\label{eq:CoarseCondition}
  C_*C_\sharp\gamma^{2m}\leq \frac14.
\end{equation}
\par
 We now define
    $\mu_{k}=\|E_{k_*+k}\|/(C_\flat m^{-1})$
 and observe that
\begin{equation*}
  \mu_k\leq 1+(C_*C_\flat m^{-1})\mu_{k-1}^2
  \qquad\text{for}\quad k\geq 1
\end{equation*}
 by \eqref{eq:FineLevels}.  It then follows from Lemma~\ref{lem:AbstractBdd}
 that
\begin{equation*}
  \mu_k\leq 2+4^{1-2^k}\mu_0 \qquad\text{for}\quad k\geq 1,
\end{equation*}
 or equivalently
\begin{equation*}
  \|E_k\|\leq 2C_\flat m^{-1}+
     4^{1-2^{k-k_*}}\|E_{k_*}\| \qquad\text{for}\quad k\geq k_*+1,
\end{equation*}
 provided that
\begin{equation*}
  C_*C_\flat m^{-1}\leq \frac{1}{4(1+\|E_{k_*}\|/(C_\flat m^{-1}))},
\end{equation*}
 or equivalently
\begin{equation}\label{eq:FineCondition}
  C_*C_\flat m^{-1}+C_*\|E_{k_*}\|\leq \frac14.
\end{equation}
\par
 Finally we observe that if we choose $m_*$ so that
   $$C_*C_\flat m_*^{-1}+2C_*C_\sharp \gamma^{2m_*}\leq \frac 14,$$
 then \eqref{eq:CoarseCondition} and \eqref{eq:FineCondition} are satisfied for
 $m\geq m_*$.
\end{proof}
\begin{remark}\label{rem:Transition}
 According to Theorem~\ref{thm:WCycle}, the  $k$-th level symmetric
 $W$-cycle algorithm is a contraction if the number of smoothing steps is sufficiently
 large and the contraction number is bounded away from $1$ uniformly in $k$ and $\beta$.
  Moreover,  for the coarser levels where $\beta^\frac12 h_k^{-2}<1$,
 the contraction number of the symmetric $W$-cycle algorithm will decrease exponentially
  with respect to the number $m$ of smoothing steps.  After a few transition levels
  the dominant term on the right-hand side of \eqref{eq:FineEst} becomes
   $2C_\flat m^{-1}$ and
   the contraction number will decrease at the rate of
  $m^{-1}$ for the finer levels where $\beta^\frac12 h_k^{-2}\geq 1$.
\end{remark}
\begin{remark}\label{rem:Nonsymmetric}
  For the nonsymmetric $W$-cycle algorithm with $m_1$ (resp., $m_2$) pre-smoothing
  (reps., post-smoothing) steps, the estimates \eqref{eq:CoarseEst} and \eqref{eq:FineEst}
  are replaced by
\begin{alignat*}{3}
  \|E_k\|&\leq 2C_\sharp \gamma^{m_1+m_2}   &\qquad&\forall\,1\leq k\leq k_*,  \\\
  \|E_k\|&\leq 2C_\flat [\max(1,m_1)\max(1,m_2)]^{-\frac12}+ 4^{1-2^{k-k_*}}(2C_\sharp \gamma^{m_1+m_2})
     &\qquad&\forall\,k\geq k_*+1.
\end{alignat*}
\end{remark}
\section{Numerical Results}\label{sec:Numerics}
 In this section we report numerical results of the symmetric $W$-cycle and $V$-cycle algorithms
  in two
 and three dimensional convex domains for $\beta=10^{-2}, 10^{-4}$ and $10^{-6}$, where
 the preconditioner $\fC_k^{-1}$  is based on a $V(4,4)$ multigrid solve   for  \eqref{eq:BVP}.
 \par
 The norm $\|E_k\|$ of the error propagation operator  is determined by a power iteration, and we employed
  the MATLAB/C++ toolbox FELICITY \cite{Walker:2018:FELICITY}
 in our computation.
\begin{example} {\bf(Unit Square)}\label{example:Square}
\par\smallskip\noindent
 The domain $\O$ for this example is the unit square $(0,1)\times(0,1)$.  The initial triangulation
 $\cT_0$ is depicted in Figure~\ref{figure:2D} and the triangulations
  $\cT_1,\ldots,\cT_7$ are generated by uniform subdivisions.
 The norms $\|E_k\|$ for the error propagation operators of the
 $k$-th level  symmetric $W$-cycle algorithm
 with $\beta=10^{-2}$  (resp., $\beta=10^{-4}$ and $\beta=10^{-6}$)
 are presented in Table~\ref{table:WCycle11}  (resp., Table~\ref{table:WCycle12} and Table~\ref{table:WCycle13}),
 where the number $m$ of pre-smoothing and post-smoothing steps increases from $2^0$ to $2^8$.
 The times for one iteration of the $W$-cycle algorithm at level 7
 (where the number of degrees of freedom (DOF) is roughly $6\times 10^4$) are also included.
\par
 We observe that the symmetric $W$-cycle algorithm is a contraction with $m=1$ for all three choices of
 $\beta$, and the behavior of the contraction numbers as $k$ and $m$ vary
  agree with Remark~\ref{rem:Transition}.  The robustness of $\|E_k\|$ with respect to $\beta$ and $k$ is
  also clearly observed.  The  times for one iteration of the $W$-cycle at level 7
   are proportional to
  the number of smoothing steps, which confirms that this is an $O(n)$ algorithm.
\end{example}
\begin{table}[htbp]
\begin{tabular}{|c|c|c|c|c|c|c|c||c|}\hline
\backslashbox{$m$}{\lower 2pt\hbox{$k$}}&1&2&3&4&5&6&7&Time \\
\hline
\rule{0pt}{2.5ex}$2^0$&3.00e-01&6.88e-01&6.70e-01&6.31e-01&6.19e-01&6.16e-01&6.15e-01&2.59e-01\\
\hline
\rule{0pt}{2.5ex}$2^1$&8.90e-02&4.83e-01&4.71e-01&4.45e-01&4.35e-01&4.32e-01&4.31e-01&5.47e-01\\
\hline
\rule{0pt}{2.5ex}$2^2$&7.93e-03&2.58e-01&2.67e-01&2.67e-01&2.63e-01&2.61e-01&2.61e-01&1.21e\splus00\\
\hline
\rule{0pt}{2.5ex}$2^3$&6.28e-05&1.09e-01&1.25e-01&1.20e-01&1.15e-01&1.14e-01&1.14e-01&1.82e\splus00\\
\hline
\rule{0pt}{2.5ex}$2^4$&3.94e-09&5.19e-02&4.55e-02&4.94e-02&5.15e-02&5.17e-02&5.18e-02&3.68e\splus00\\
\hline
\rule{0pt}{2.5ex}$2^5$&1.41e-16&2.01e-02&1.57e-02&2.45e-02&2.48e-02&2.55e-02&2.56e-02&7.07e\splus00\\
\hline
\rule{0pt}{2.5ex}$2^6$&2.83e-17&3.11e-03&7.97e-03&9.39e-03&1.22e-02&1.30e-02&1.31e-02&1.42e\splus01\\
\hline
\rule{0pt}{2.5ex}$2^7$&-&7.45e-05&2.09e-03&3.53e-03&5.98e-03&6.22e-03&6.43e-03&2.83e\splus01\\
\hline
\rule{0pt}{2.5ex}$2^8$&-&4.27e-08&1.44e-04&1.67e-03&2.23e-03&3.05e-03&3.26e-03&5.51e\splus01\\
\hline
\end{tabular}
\par\bigskip
 \caption{The norm for the error propagation operator of the
 symmetric $W$-cycle algorithm
 with $\beta=10^{-2}$,
  together with the time (in seconds) for one iteration
 of the $W$-cycle algorithm at level 7 ($\O=$ unit square)}
\label{table:WCycle11}
\end{table}
\begin{table}[htbp]
\begin{tabular}{|c|c|c|c|c|c|c|c||c|}\hline
\backslashbox{$m$}{\lower 2pt\hbox{$k$}}&1&2&3&4&5&6&7&Time \\
\hline
\rule{0pt}{2.5ex}$2^0$&7.01e-02&2.20e-01&6.50e-01&7.10e-01&6.56e-01&6.24e-01&6.17e-01&2.66e-01\\
\hline
\rule{0pt}{2.5ex}$2^1$&6.09e-03&6.31e-02&6.62e-01&5.44e-01&4.68e-01&4.40e-01&4.33e-01&5.67e-01\\
\hline
\rule{0pt}{2.5ex}$2^2$&3.71e-05&3.99e-03&5.62e-01&3.57e-01&2.88e-01&2.67e-01&2.62e-01&1.13e\splus00\\
\hline
\rule{0pt}{2.5ex}$2^3$&1.38e-09&1.59e-05&4.08e-01&1.91e-01&1.34e-01&1.19e-01&1.15e-01&1.98e\splus00\\
\hline
\rule{0pt}{2.5ex}$2^4$&2.10e-17&5.00e-09&2.17e-01&6.83e-02&5.78e-02&5.34e-02&5.22e-02&3.62e\splus00\\
\hline
\rule{0pt}{2.5ex}$2^5$&9.25e-20&1.85e-17&6.14e-02&2.78e-02&2.91e-02&2.67e-02&2.59e-02&7.17e\splus00\\
\hline
\rule{0pt}{2.5ex}$2^6$&8.90e-17&4.40e-18&4.93e-03&6.51e-03&1.31e-02&1.37e-02&1.34e-02&1.42e\splus01\\
\hline
\rule{0pt}{2.5ex}$2^7$&8.02e-18&5.53e-18&3.19e-05&1.21e-03&4.31e-03&6.78e-03&6.64e-03&2.97e\splus01\\
\hline
\rule{0pt}{2.5ex}$2^8$&8.90e-17&5.35e-17&1.33e-09&4.19e-05&9.69e-04&3.02e-03&3.37e-03&5.70e\splus01\\
\hline
\end{tabular}
\par\bigskip
 \caption{The norm for the error propagation operator of the
 symmetric $W$-cycle algorithm
 with $\beta=10^{-4}$,
 together with the time (in seconds) for one iteration
 of the $W$-cycle algorithm at level 7 ($\O=$ unit square)}
\label{table:WCycle12}
 \end{table}
\begin{table}[htbp]
\begin{tabular}{|c|c|c|c|c|c|c|c||c|}\hline
\backslashbox{$m$}{\lower 2pt\hbox{$k$}}&1&2&3&4&5&6&7&Time \\
\hline
\rule{0pt}{2.5ex}$2^0$&2.55e-01&4.38e-01&3.89e-01&5.68e-01&8.92e-01&8.66e-01&8.36e-01&2.68e-01\\
\hline
\rule{0pt}{2.5ex}$2^1$&6.47e-02&1.93e-01&1.54e-01&3.90e-01&8.13e-01&7.72e-01&7.27e-01&5.32e-01\\
\hline
\rule{0pt}{2.5ex}$2^2$&4.33e-03&3.86e-02&2.51e-02&1.87e-01&7.01e-01&6.37e-01&5.89e-01&9.41e-01\\
\hline
\rule{0pt}{2.5ex}$2^3$&1.96e-05&1.54e-03&6.95e-04&4.06e-02&5.64e-01&4.79e-01&4.40e-01&1.84e\splus00\\
\hline
\rule{0pt}{2.5ex}$2^4$&4.01e-10&2.44e-06&4.89e-07&1.85e-03&4.04e-01&2.96e-01&2.52e-01&3.60e\splus00\\
\hline
\rule{0pt}{2.5ex}$2^5$&7.45e-17&6.23e-12&2.40e-13&4.22e-06&2.12e-01&1.35e-01&1.19e-01&7.16e\splus00\\
\hline
\rule{0pt}{2.5ex}$2^6$&1.08e-16&6.23e-18&3.79e-18&2.11e-11&5.92e-02&6.24e-02&6.16e-02&1.44e\splus01\\
\hline
\rule{0pt}{2.5ex}$2^7$&1.52e-18&1.51e-17&4.83e-17&2.58e-17&7.81e-03&2.08e-02&2.88e-02&2.80e\splus01\\
\hline
\rule{0pt}{2.5ex}$2^8$&1.85e-16&5.20e-17&7.33e-18&2.23e-17&2.14e-04&3.55e-03&1.34e-02&5.56e\splus01\\
\hline
\end{tabular}
\par\bigskip
 \caption{The norms for the error propagation operator of the
  symmetric $W$-cycle algorithm
 with $\beta=10^{-6}$,
 together with the time (in seconds) for one iteration
 of the $W$-cycle algorithm at level 7 ($\O=$ unit square)}
\label{table:WCycle13}
\end{table}
\par
 We have also computed the norms of the error propagation operators for the $k$-th level
 symmetric $V$-cycle algorithm,
 which are very similar to those of the $W$-cycle algorithm.  For brevity we only
 present the  results for
 $\beta=10^{-2},10^{-4},10^{-6}$ and $m=2^0,2^1,2^2$ in Table~\ref{table:VCycle1}.
 Again we observe that the $V$-cycle algorithm is a contraction for $m=1$ and the contraction numbers
 are robust with respect to both $\beta$ and $k$.
\begin{table}[htbp]
\begin{tabular}{|c|c|c|c|c|c|c|c||c|}\hline
\backslashbox{$m$}{\lower 2pt\hbox{$k$}}&1&2&3&4&5&6&7&Time \\
\hline
\multicolumn{9}{|c|}{\rule{0pt}{2.5ex}$\beta=10^{-2}$}\\[2pt]
\hline
\rule{0pt}{2.5ex}$2^0$&3.00e-01&6.88e-01&7.02e-01&6.99e-01&7.00e-01&7.02e-01&7.03e-01&7.65e-02\\
\hline
\rule{0pt}{2.5ex}$2^1$&8.90e-02&4.84e-01&5.13e-01&5.11e-01&5.13e-01&5.15e-01&5.18e-01&1.15e-01\\
\hline
\rule{0pt}{2.5ex}$2^2$&7.93e-03&2.58e-01&2.96e-01&3.06e-01&3.10e-01&3.14e-01&3.17e-01&2.03e-01\\
\hline
\multicolumn{9}{|c|}{\rule{0pt}{2.5ex}$\beta=10^{-4}$}\\[2pt]
\hline
\rule{0pt}{2.5ex}$2^0$&7.01e-02&2.17e-01&6.74e-01&7.08e-01&7.06e-01&7.07e-01&7.09e-01&6.30e-02\\
\hline
\rule{0pt}{2.5ex}$2^1$&6.09e-03&6.31e-02&6.72e-01&5.42e-01&5.26e-01&5.27e-01&5.28e-01&1.08e-01\\
\hline
\rule{0pt}{2.5ex}$2^2$&3.71e-05&3.99e-03&5.63e-01&3.57e-01&3.51e-01&3.49e-01&3.49e-01&1.98e-01\\
\hline
\multicolumn{9}{|c|}{\rule{0pt}{2.5ex}$\beta=10^{-6}$}\\[2pt]
\hline
\rule{0pt}{2.5ex}$2^0$&2.55e-01&4.40e-01&3.88e-01&5.78e-01&8.94e-01&8.93e-01&8.94e-01&6.01e-02\\
\hline
\rule{0pt}{2.5ex}$2^1$&6.47e-02&1.94e-01&1.54e-01&3.86e-01&8.15e-01&8.11e-01&8.14e-01&1.06e-01\\
\hline
\rule{0pt}{2.5ex}$2^2$&4.33e-03&3.87e-02&2.51e-02&1.86e-01&7.02e-01&6.89e-01&6.95e-01&2.15e-01\\
\hline
\end{tabular}
\par\bigskip
 \caption{The norm for the error propagation operator of the
symmetric $V$-cycle algorithm,
 together with the time (in seconds) for one iteration of the
 $V$-cycle algorithm at level 7 ($\O=$ unit square)}
\label{table:VCycle1}
\end{table}
\begin{figure}[htbp]
\hfill
\begin{minipage}{.3\textwidth}
  \centering
  \includegraphics[width=.5\linewidth]{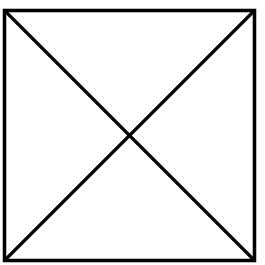}
\end{minipage}%
\begin{minipage}{.5\textwidth}
  \centering
  \includegraphics[width=.3\linewidth]{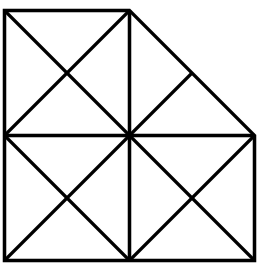}
\end{minipage}
\hfill
\caption{Initial triangulations for  the unit square and the pentagonal domain}
\label{figure:2D}
\end{figure}
\begin{remark}\label{rem:V44}
  We include the contract numbers for $m$ up to $2^8$ in
  Table~\ref{table:WCycle11}--Table~\ref{table:WCycle13}
  so that the theoretical error estimates in Theorem~\ref{thm:WCycle} are clearly visible.
  If we use instead a $V(8,8)$  multigrid solve for \eqref{eq:BVP} in the construction of
  the preconditioner $\fC_k^{-1}$, then it would be enough to show the results for $m$ up to $2^6$.
  This is also true for the other examples.
\end{remark}
\begin{example} {\bf(Pentagonal Domain)}
\par\smallskip\noindent
 The domain $\O$ for this example is the convex pentagonal domain obtained from the unit square by
 removing the triangle with vertices $(1,0.5)$, $(1,1)$ and $(0.5,1)$.  The initial
 triangulation $\cT_0$ is depicted in Figure~\ref{figure:2D} and the triangulations
 $\cT_1,\ldots,\cT_6$ are generated by uniform subdivisions.
\par
 The performance of the symmetric $W$-cycle (and $V$-cycle) algorithm is similar to that for
 the unit square.  In Table~\ref{table:WCycle2} we only report the numerical results for the $W$-cycle algorithm
 with $m=2^0,2^1,2^2$.
 Again we observe that the $W$-cycle algorithm is
 a contraction for $m=1$ and $\|E_k\|$ is robust with respect to $\beta$ and $k$.
 The contraction numbers in Table~\ref{table:WCycle2} are similar to the corresponding contraction numbers
 in Tables~\ref{table:WCycle11}--\ref{table:WCycle13}.
 (The number of DOF at refinement level 6 is roughly $5.7\times10^4$.)
\end{example}
\begin{table}[htbp]
\begin{tabular}{|c|c|c|c|c|c|c||c|}\hline
\backslashbox{$m$}{\lower 2pt\hbox{$k$}}&1&2&3&4&5&6&Time\\
\hline
\multicolumn{8}{|c|}{\rule{0pt}{2.5ex}$\beta=10^{-2}$}\\[2pt]
\hline
\rule{0pt}{2.5ex}$2^0$&6.56e-01&6.69e-01&6.33e-01&6.25e-01&6.23e-01&6.23e-01&1.58e-01\\
\hline
\rule{0pt}{2.5ex}$2^1$&4.68e-01&4.79e-01&4.41e-01&4.27e-01&4.24e-01&4.23e-01&2.76e-01\\
\hline
\rule{0pt}{2.5ex}$2^2$&3.68e-01&2.69e-01&2.37e-01&2.20e-01&2.11e-01&2.05e-01&5.23e-01\\
\hline
\multicolumn{8}{|c|}{\rule{0pt}{2.5ex}$\beta=10^{-4}$}\\[2pt]
\hline
\rule{0pt}{2.5ex}$2^0$&2.42e-01&6.94e-01&7.22e-01&6.58e-01&6.25e-01&6.23e-01&1.56e-01\\
\hline
\rule{0pt}{2.5ex}$2^1$&6.01e-02&5.36e-01&5.60e-01&4.61e-01&4.30e-01&4.25e-01&2.94e-01\\
\hline
\rule{0pt}{2.5ex}$2^2$&4.83e-03&3.12e-01&3.39e-01&2.47e-01&2.17e-01&2.07e-01&5.23e-01\\
\hline
\multicolumn{8}{|c|}{\rule{0pt}{2.5ex}$\beta=10^{-6}$}\\[2pt]
\hline
\rule{0pt}{2.5ex}$2^0$&3.83e-01&2.99e-01&6.34e-01&8.73e-01&8.52e-01&8.22e-01&1.52e-01\\
\hline
\rule{0pt}{2.5ex}$2^1$&1.09e-01&9.60e-02&4.22e-01&8.30e-01&7.83e-01&7.22e-01&2.80e-01\\
\hline
\rule{0pt}{2.5ex}$2^2$&2.26e-02&9.63e-03&1.89e-01&7.41e-01&6.38e-01&5.65e-01&5.24e-01\\
\hline
\end{tabular}
\par\bigskip
\caption{The norm for the error propagation operator of the
 symmetric $W$-cycle algorithm,
together with the time (in seconds) for one iteration of the
 $W$-cycle algorithm at level 6 ($\O=$ pentagonal domain)}
\label{table:WCycle2}
\end{table}
\newpage
\begin{example}  {\bf(Unit Cube)} \label{example:Cube}
\par\smallskip\noindent
 The domain for this example is the unit cube $(0,1)\times(0,1)\times(0,1)$.
 The triangulations $\cT_0$ and $\cT_1$ are depicted in Figure~\ref{figure:Cube}.
 The number of grid points in all directions are doubled in each refinement and the
 triangulations inside the cubic subdomains at all levels are similar to one and other.
\begin{figure}[htbp]
\begin{minipage}{.35\textwidth}
  \centering
  \includegraphics[width=1\linewidth]{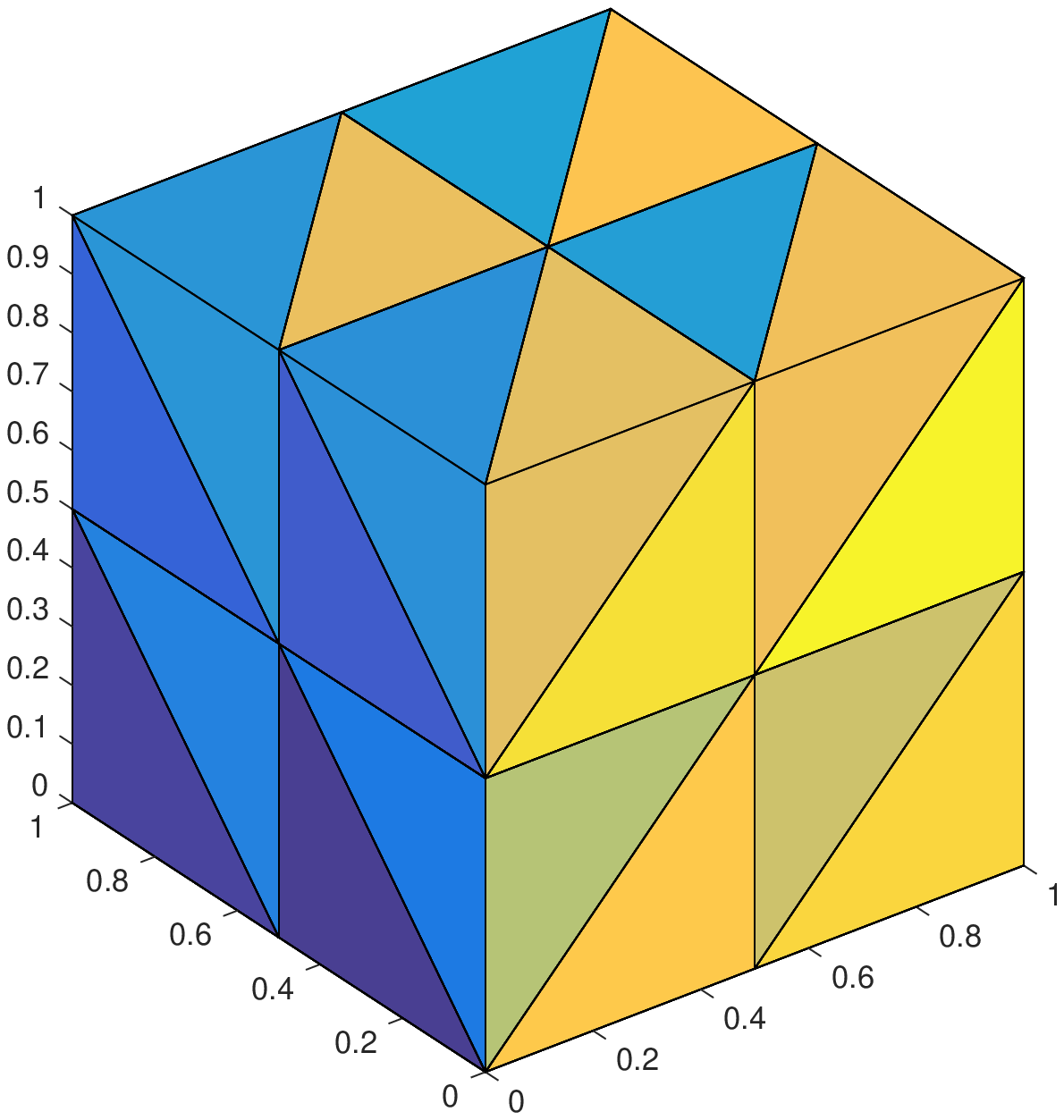}
\end{minipage}%
\begin{minipage}{.35\textwidth}
  \centering
  \includegraphics[width=1\linewidth]{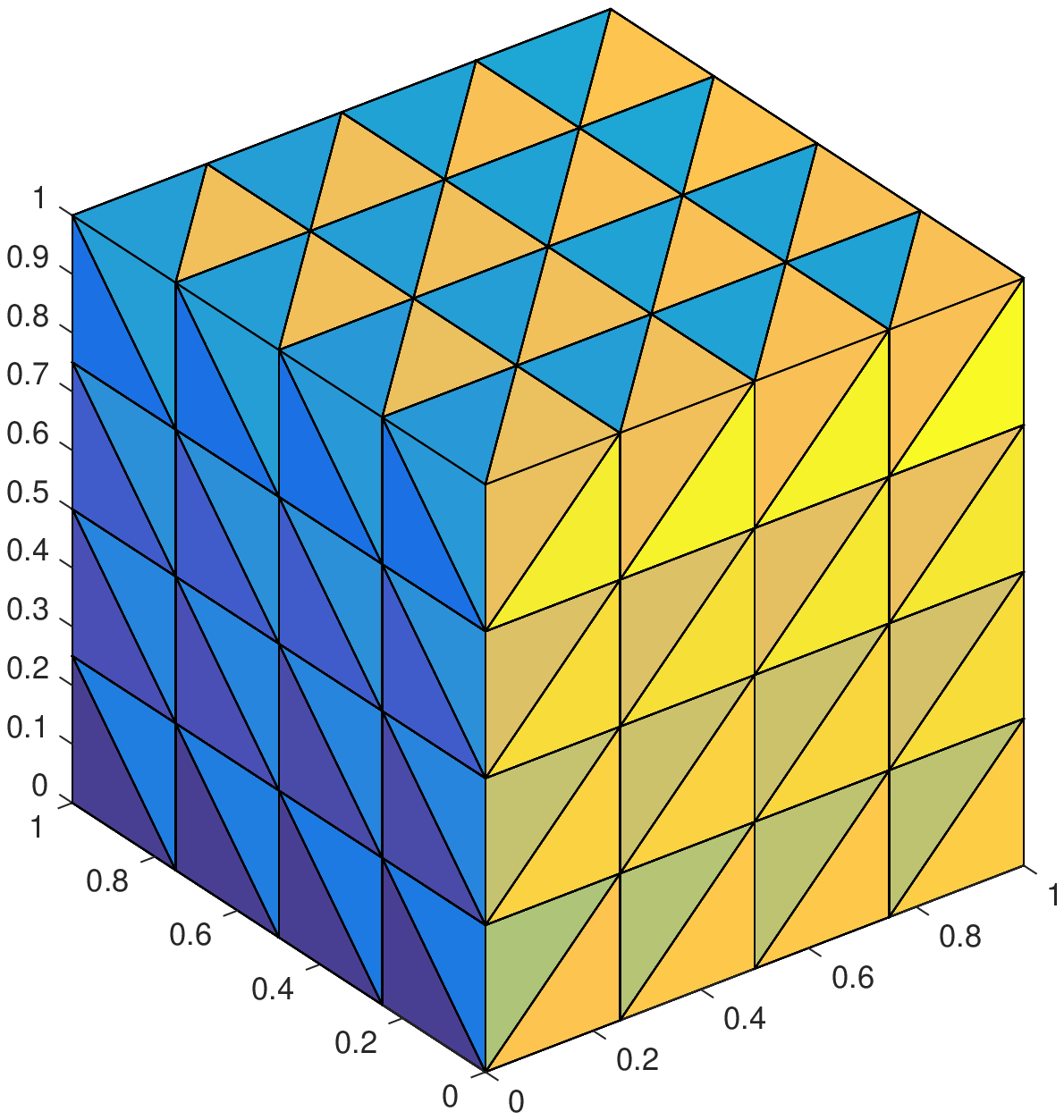}
\end{minipage}
\caption{Triangulations $\cT_0$ and $\cT_1$ for the unit cube}
\label{figure:Cube}
\end{figure}
\par
 The norms $\|E_k\|$ for the error propagation operators of the
 symmetric $W$-cycle algorithm
 with $\beta=10^{-2}$  (resp., $\beta=10^{-4}$ and $\beta=10^{-6}$)
 are displayed in Table~\ref{table:WCycle31}  (resp., Table~\ref{table:WCycle32} and Table~\ref{table:WCycle33}),
 where the number $m$ of pre-smoothing and post-smoothing steps increases from $2^0$ to $2^8$.
 The times for one iteration of the $W$-cycle algorithm at level 5
 (where the number of DOF is roughly $5\times 10^5$) are also included.
\begin{table}[htbp]
\begin{tabular}{|c|c|c|c|c|c||c|}\hline
\backslashbox{$m$}{\lower 2pt\hbox{$k$}}&1&2&3&4&5&Time\\
\hline
\rule{0pt}{2.5ex}$2^0$&3.25e-01&4.12e-01&6.94e-01&8.58e-01&9.04e-01&5.22e-01\\
\hline
\rule{0pt}{2.5ex}$2^1$&1.66e-01&2.16e-01&5.27e-01&7.92e-01&8.85e-01&9.33e-01\\
\hline
\rule{0pt}{2.5ex}$2^2$&5.06e-02&1.51e-01&3.27e-01&6.69e-01&8.56e-01&1.78e\splus00\\
\hline
\rule{0pt}{2.5ex}$2^3$&4.83e-03&8.28e-02&1.90e-01&4.92e-01&7.84e-01&3.98e\splus00\\
\hline
\rule{0pt}{2.5ex}$2^4$&4.41e-05&3.12e-02&1.13e-01&2.89e-01&6.59e-01&6.96e\splus00\\
\hline
\rule{0pt}{2.5ex}$2^5$&3.68e-09&1.03e-02&5.53e-02&1.71e-01&4.81e-01&1.38e\splus01\\
\hline
\rule{0pt}{2.5ex}$2^6$&9.23e-17&4.74e-03&2.89e-02&9.65e-02&2.73e-01&2.77e\splus01\\
\hline
\rule{0pt}{2.5ex}$2^7$&-&2.30e-05&8.31e-03&5.16e-02&1.67e-01&5.48e\splus01\\
\hline
\rule{0pt}{2.5ex}$2^8$&-&5.28e-10&7.67e-04&2.45e-02&9.64e-02&1.10e\splus02\\
\hline
\end{tabular}
\par\bigskip
\caption{The norm for the error propagation operator of the
 symmetric $W$-cycle algorithm
 with  $\beta=10^{-2}$,
 together with the time (in seconds) for one iteration of the
 $W$-cycle algorithm at level 5 ($\O=$ unit cube)}
\label{table:WCycle31}
\end{table}
\begin{table}[htbp]
\begin{tabular}{|c|c|c|c|c|c||c|}\hline
\backslashbox{$m$}{\lower 2pt\hbox{$k$}}&1&2&3&4&5&Time\\
\hline
\rule{0pt}{2.5ex}$2^0$&1.67e-01&4.90e-01&6.50e-01&7.41e-01&8.58e-01&5.43e-01\\
\hline
\rule{0pt}{2.5ex}$2^1$&2.97e-02&2.89e-01&4.74e-01&6.03e-01&8.15e-01&9.37e-01\\
\hline
\rule{0pt}{2.5ex}$2^2$&9.87e-04&1.24e-01&2.85e-01&4.20e-01&7.18e-01&1.79e\splus00\\
\hline
\rule{0pt}{2.5ex}$2^3$&1.07e-06&2.62e-02&1.49e-01&2.48e-01&5.70e-01&4.22e\splus00\\
\hline
\rule{0pt}{2.5ex}$2^4$&2.53e-10&1.14e-03&5.84e-02&1.49e-01&3.64e-01&7.86e\splus00\\
\hline
\rule{0pt}{2.5ex}$2^5$&8.22e-17&2.43e-06&1.57e-02&7.39e-02&2.21e-01&1.40e\splus01\\
\hline
\rule{0pt}{2.5ex}$2^6$&1.36e-18&1.05e-11&7.38e-04&2.27e-02&1.30e-01&2.76e\splus01\\
\hline
\rule{0pt}{2.5ex}$2^7$&2.31e-17&9.11e-17&1.70e-06&4.88e-03&7.21e-02&5.51e\splus01\\
\hline
\rule{0pt}{2.5ex}$2^8$&2.10e-17&3.02e-17&9.36e-12&1.66e-04&3.21e-02&1.10e\splus02\\
\hline
\end{tabular}
\par\bigskip
\caption{The norm for the error propagation operator of the
 symmetric $W$-cycle algorithm
 with  $\beta=10^{-4}$,
 together with the time (in seconds) for one iteration of the
 $W$-cycle algorithm at level 5 ($\O=$ unit cube)}
\label{table:WCycle32}
\end{table}
\begin{table}[htbp]
\begin{tabular}{|c|c|c|c|c|c||c|}\hline
\backslashbox{$m$}{\lower 2pt\hbox{$k$}}&1&2&3&4&5&Time\\
\hline
\rule{0pt}{2.5ex}$2^0$&2.52e-01&3.15e-01&7.92e-01&8.96e-01&8.75e-01&5.10e-01\\
\hline
\rule{0pt}{2.5ex}$2^1$&6.57e-02&1.26e-01&6.38e-01&8.62e-01&7.94e-01&9.50e-01\\
\hline
\rule{0pt}{2.5ex}$2^2$&4.81e-03&1.89e-02&4.25e-01&7.76e-01&6.57e-01&1.78e\splus00\\
\hline
\rule{0pt}{2.5ex}$2^3$&4.62e-05&1.20e-03&1.96e-01&6.36e-01&4.70e-01&4.04e\splus00\\
\hline
\rule{0pt}{2.5ex}$2^4$&3.01e-09&1.87e-06&4.15e-02&4.37e-01&2.76e-01&7.75e\splus00\\
\hline
\rule{0pt}{2.5ex}$2^5$&1.86e-17&3.54e-12&2.19e-03&2.12e-01&1.21e-01&1.38e\splus01\\
\hline
\rule{0pt}{2.5ex}$2^6$&2.74e-17&3.42e-17&2.47e-05&5.24e-02&3.79e-02&2.77e\splus01\\
\hline
\rule{0pt}{2.5ex}$2^7$&3.53e-18&7.74e-17&9.88e-10&6.76e-03&7.38e-03&5.55e\splus01\\
\hline
\rule{0pt}{2.5ex}$2^8$&8.94e-17&1.14e-17&3.25e-16&1.64e-04&1.61e-04&1.10e\splus02\\
\hline
\end{tabular}
\par\bigskip
\caption{The norm for the error propagation operator of the
  symmetric $W$-cycle algorithm
 with  $\beta=10^{-6}$,
 together with the time (in seconds) for one iteration of the
 $W$-cycle algorithm at level 5 ($\O=$ unit cube)}
\label{table:WCycle33}
\end{table}
\par
 We  observe that the symmetric $W$-cycle algorithm is a contraction for $m=1$.  The behavior
 of the contraction numbers agree with Remark~\ref{rem:Transition} and they are robust with respect to
 both $\beta$ and $k$.
 The performance of the symmetric $V$-cycle algorithm is similar and we only present the numerical results
 for $m=2^0$, $2^1$ and $2^{2}$ in Table~\ref{table:VCycle2}.
\begin{table}[hh]
\begin{tabular}{|c|c|c|c|c|c||c|}\hline
\backslashbox{$m$}{\lower 2pt\hbox{$k$}}&1&2&3&4&5&Time\\
\hline
\multicolumn{7}{|c|}{\rule{0pt}{2.5ex}$\beta=10^{-2}$}\\[2pt]
\hline
\rule{0pt}{2.5ex}$2^0$&3.25e-01&4.45e-01&6.46e-01&7.99e-01&8.55e-01&4.27e-01\\
\hline
\rule{0pt}{2.5ex}$2^1$&1.66e-01&2.77e-01&4.87e-01&7.49e-01&8.47e-01&7.72e-01\\
\hline
\rule{0pt}{2.5ex}$2^2$&5.06e-02&1.69e-01&2.95e-01&6.40e-01&8.17e-01&1.49e\splus00\\
\hline
\multicolumn{7}{|c|}{\rule{0pt}{2.5ex}$\beta=10^{-4}$}\\[2pt]
\hline
\rule{0pt}{2.5ex}$2^0$&1.67e-01&4.88e-01&6.44e-01&7.37e-01&8.18e-01&4.25e-01\\
\hline
\rule{0pt}{2.5ex}$2^1$&2.97e-02&2.89e-01&4.65e-01&6.00e-01&7.62e-01&8.07e-01\\
\hline
\rule{0pt}{2.5ex}$2^2$&9.87e-04&1.24e-01&2.78e-01&4.20e-01&6.71e-01&1.49e\splus00\\
\hline
\multicolumn{7}{|c|}{\rule{0pt}{2.5ex}$\beta=10^{-6}$}\\[2pt]
\hline
\rule{0pt}{2.5ex}$2^0$&2.52e-01&3.13e-01&7.92e-01&8.88e-01&8.86e-01&4.39e-01\\
\hline
\rule{0pt}{2.5ex}$2^1$&6.57e-02&1.26e-01&6.38e-01&8.47e-01&8.19e-01&7.87e-01\\
\hline
\rule{0pt}{2.5ex}$2^2$&4.81e-03&1.89e-02&4.25e-01&7.70e-01&7.05e-01&1.66e\splus00\\
\hline
\end{tabular}
\par\bigskip
\caption{The norm for the error propagation operator of the
 symmetric $V$-cycle algorithm,
 together with the time (in seconds) for one iteration of the
 $V$-cycle algorithm at level 5 ($\O=$ unit cube)}
\label{table:VCycle2}
\end{table}
\end{example}
\section{Concluding Remarks}\label{sec:Conclusions}
 In this paper
 we construct multigrid  algorithms for a model linear-quadratic
 elliptic optimal control problem and prove
 that for convex domains the $W$-cycle algorithm with a sufficiently large
 number of smoothing steps is uniformly convergent with respect to mesh refinements
 and a regularizing parameter.
  The theoretical estimates and the performance of the algorithms
  are demonstrated by numerical results.
\par
 For the numerical results in Section~\ref{sec:Numerics},
 we use a $V(4,4)$ multigrid solve for \eqref{eq:BVP}
  in the construction of the preconditioner $\fC_k^{-1}$.  But in fact
  the symmetric $V$-cycle multigrid algorithm from Section~\ref{subsec:VCycle}
   based on a $V(1,1)$ solve for \eqref{eq:BVP}
    also converges uniformly with $1$ pre-smoothing step and $1$ post-smoothing step.
 The results for the unit square and
 unit cube are reported in Table~\ref{table:SquareV11} and Table~\ref{table:CubeV11}.
\begin{table}[htbp]
\begin{tabular}{|c|c|c|c|c|c|c|c||c|}\hline
\backslashbox{$\beta$}{\lower 2pt\hbox{$k$}}&1&2&3&4&5&6&7&Time\\
\hline
\rule{0pt}{2.5ex}$10^{-2}$&2.15e-01&7.92e-01&8.20e-01&8.24e-01&8.29e-01&8.33e-01&8.36e-01&4.99e-02\\
\hline
\rule{0pt}{2.5ex}$10^{-4}$&6.32e-02&2.83e-01&6.25e-01&8.11e-01&8.23e-01&8.30e-01&8.35e-01&4.59e-02\\
\hline
\rule{0pt}{2.5ex}$10^{-6}$&2.56e-01&6.49e-01&7.22e-01&8.36e-01&9.52e-01&9.56e-01&9.58e-01&4.58e-02\\
\hline
\end{tabular}
\par\bigskip
\caption{The norm for the error propagation operator of the
  symmetric $V$-cycle algorithm from Section~\ref{subsec:VCycle}
  with $m=1$,
 together with the time (in seconds) for one iteration of the
 $V$-cycle algorithm at level 7
 ($\O=$ unit square)}
\label{table:SquareV11}
\end{table}
\begin{table}[htbp]
\begin{tabular}{|c|c|c|c|c|c||c|}\hline
\backslashbox{$\beta$}{\lower 2pt\hbox{$k$}}&1&2&3&4&5&Time\\
\hline
\rule{0pt}{2.5ex}$10^{-2}$&3.91e-01&6.63e-01&8.16e-01&8.71e-01&8.92e-01&3.21e-01\\
\hline
\rule{0pt}{2.5ex}$10^{-4}$&1.72e-01&7.17e-01&8.42e-01&8.84e-01&9.05e-01&3.18e-01\\
\hline
\rule{0pt}{2.5ex}$10^{-6}$&2.70e-01&5.76e-01&9.31e-01&9.47e-01&9.47e-01&3.15e-01\\
\hline
\end{tabular}
\par\bigskip
\caption{The norm for the error propagation operator of the
  symmetric $V$-cycle algorithm from Section~\ref{subsec:VCycle}
 with  $m=1$, together with the time (in seconds) for one iteration of the
 $V$-cycle algorithm at level 5 ($\O=$ unit cube)}
\label{table:CubeV11}
\end{table}
\par
  Moreover numerical results indicate that
   our multigrid algorithms are also robust for nonconvex domains.  The results for the symmetric
   $V$-cycle algorithm with 1 pre-smoothing step and 1 post-smoothing step
  can be found in Table~\ref{table:LShapeV11},  where the preconditioner
  is also based on a $V(1,1)$ solve for  \eqref{eq:BVP}.
  (The number of DOF at level 6 is roughly $4.8\times10^4$.)
   However our theory for the convex domain does not immediately generalize to nonconvex domains.
  Note that nonconvex domains have been treated in \cite{TZ:2013:MGOCP} with respect to an abstract norm
  defined through the interpolation between function spaces.
\begin{table}[hh]
\begin{tabular}{|c|c|c|c|c|c|c||c|}\hline
\backslashbox{$\beta$}{\lower 2pt\hbox{$k$}}&1&2&3&4&5&6&Time\\
\hline
\rule{0pt}{2.5ex}$10^{-2}$&8.36e-01&8.19e-01&8.21e-01&8.36e-01&8.59e-01&8.89e-01&3.81e-02\\
\hline
\rule{0pt}{2.5ex}$10^{-4}$&1.86e-01&3.54e-01&6.81e-01&7.17e-01&7.25e-01&7.29e-01&3.84e-02\\
\hline
\rule{0pt}{2.5ex}$10^{-6}$&4.74e-01&5.07e-01&7.07e-01&8.67e-01&8.91e-01&9.07e-01&3.82e-02\\
\hline
\end{tabular}
\par\bigskip
\caption{The norm for the error propagation operator of the
 symmetric $V$-cycle algorithm from Section~\ref{subsec:VCycle}
 with $m=1$,
 together with the time (in seconds) for one iteration of the
 $V$-cycle algorithm at level 7
 ($\O=$ $L$-shaped domain)}
\label{table:LShapeV11}
\end{table}
\par
 One of the features of our multigrid algorithms is that they can be applied to
 nonsymmetric saddle point problems with only a trivial modification
 (cf. \cite{BLS:2017:Oseen,BOS:2018:Darcy}).  For example,
 we can also modify our multigrid algorithms to solve
  an optimal control problem with the constraint
 \eqref{eq:PDEConstraint} replaced by
\begin{equation*}
 (\nabla y,\nabla v)_\LT+(\bm{\zeta}\cdot\nabla y,v)_\LT=(u,v)_\LT \qquad
 \forall\,v\in H^1_0(\O),
\end{equation*}
 where $\bm{\zeta}\in [W^{1,\infty}(\O)]^d$ and $\nabla\cdot\bm{\zeta}=0$.
 This and the extension of our theory to nonconvex domains and the $V$-cycle algorithm  will be investigated
 in our ongoing projects.

\end{document}